\documentclass[reqno]{amsart}
\usepackage{amsfonts}
\usepackage[english]{babel}
\usepackage[utf8]{inputenc}
\newtheorem{theorem}{Theorem}
\theoremstyle{plain}

\newtheorem{definition}{Definition}

\newtheorem{lemma}{Lemma}

\newtheorem{proposition}{Proposition}
\newtheorem{remark}{Remark}

\numberwithin{equation}{section}
 \numberwithin{theorem}{section}
 \numberwithin{proposition}{section}
 \numberwithin{remark}{section}
 \numberwithin{definition}{section}
 \numberwithin{lemma}{section}
 \numberwithin{corollary}{section}
 \numberwithin{example}{section}
 \numberwithin{claim}{section}
\begin{document}
\title[Hyperbolic-parabolic equations in porous media]{%
Reiterated homogenization of hyperbolic-parabolic equations in domains with tiny holes}
\author{Hermann Douanla}
\address{Hermann Douanla, Department of Mathematics, University of Yaounde
1, P.O. Box 812, Yaounde, Cameroon}
\email{hdouanla@gmail.com}
\author{Erick Tetsadjio}
\address{Erick Tetsadjio, Department of Mathematics, University of Yaounde
1, P.O. Box 812, Yaounde, Cameroon}
\email{tetsadjio@gmail.com}
\date{October, 2016}

\begin{abstract}
The article studies the homogenization of hyperbolic-parabolic equations in porous media with tiny holes. We assume that the
holes are periodically distributed and that the coefficients of the equations are periodic. Using the multi-scale convergence method, we derive a homogenization result whose limit problem is defined on a fixed domain and is of the same type as the  problem with oscillating coefficients.
\end{abstract}

\maketitle

\section{Introduction\label{S1}}
In porous media with tiny holes, we study the asymptotic behaviour (as $\varepsilon\to 0$) of the solution to the following problem with rapidly oscillating coefficients:
\begin{equation}\label{eq1}
\left\{ \begin{aligned} \rho \left( \frac{x}{\varepsilon^2 }\right)
\frac{\partial^2 u_{\varepsilon }}{\partial t^2} + \beta \left(
\frac{x}{\varepsilon }, \frac{t}{\varepsilon^2}\right)
\frac{\partial u_{\varepsilon }}{\partial t} - \text{div}\left(
A\left(\frac{x}{\varepsilon },\frac{x}{\varepsilon ^{2}}\right)
\nabla u_{\varepsilon }\right)= & \ f \quad \text{ in }\ \
\Omega_{T}^{\varepsilon }=\Omega ^{\varepsilon }\times (0,T), \\
A\left( \frac{x}{\varepsilon },\frac{x}{\varepsilon ^{2}}\right)
\nabla u_{\varepsilon }\cdot \nu_\varepsilon = & \ 0 \quad \text{ on
}\ \ (\partial
\Omega^{\varepsilon}\setminus \partial \Omega)\times \left( 0,T\right), \\
u_{\varepsilon }= & \ 0  \quad\text{ on }\ \ (\partial \Omega^{\varepsilon}\cap
\partial \Omega) \times \left( 0,T\right), \\ u_{\varepsilon }(x,0)= & \ u^{0}(x)
 \quad \text{ in }\ \ \Omega ^{\varepsilon },\\
 \rho\left(\frac{x}{\varepsilon^2}\right)\frac{\partial u_\varepsilon}{\partial t}(x,0)= & \ \rho^\frac{1}{2}\left(\frac{x}{\varepsilon^2}\right)v^0(x)  \quad \text{ in }\ \ \Omega ^{\varepsilon }, \end{aligned}\right.
\end{equation}
where $\Omega $ is a bounded domain in $\mathbb{R}^{N}$ ($N\geq 3$) locally
located on one side of its Lipschitz boundary $\partial \Omega $, $f\in L^2(0,T;L^2(\Omega))$, $u^0\in H^1_0 (\Omega)$, $v^0\in L^2(\Omega)$,   $T>0$ is a fixed real number representing the final time of the process and $\Omega ^{\varepsilon }$ is a  domain with periodically distributed tiny holes. The coefficients $\rho$, $\beta$ and the matrix $A$ are periodic. A detailed description of the domain $\Omega^\varepsilon$ and precise assumptions  on the coefficients are given in the next section.

The equations of the form (\ref{eq1}) are usually called hyperbolic-parabolic equations (H-P equations) and appears when modelling wave processes arising for instance, in heat theory ($\rho=0$ and $\beta\neq 0$), theories of hydrodynamics, electricity, magnetism, light, sound and in elasticity theory ($\rho\neq 0$ and $\beta\neq 0$) (see e.g., \cite{evariste, steklov}). It is also well known \cite{BP, TS} that equations of the form (\ref{eq1}) model the process of small longitudinal linear elastic vibration in a thin inhomogeneous rod, in this case, $\rho\neq 0$ is the linear density of the rod,  $\beta=\beta(y)\neq 0$ the dissipation coefficient, $A$ the Young's modulus, $f$ the distribution of the density of an external force directed along the rod and $u_\varepsilon$ the displacement function.

The homogenization problem for H-P equations was first studied by Bensoussan, Lions and Papanicolau~\cite{blp} in a fixed domain by means of compactness arguments and Tatar's test function method. Bakhvalov and Panasenko~\cite{BP} considered the same problem and utilized the formal asymptotic expansion method combined with maximum principles to prove homogenization results.

To the best of our knowledge, Migorski~\cite{mgk} was the first to address the homogenization problem for H-P equations in perforated domains. In a domain perforated with holes of size $\varepsilon$, he considered a $Y$-periodic matrix $A$ and assumed some strong convergence hypotheses on $\rho^\varepsilon$ and $\beta^\varepsilon$ to prove a homogenization theorem by means of the test function method. Timofte~\cite{timofte} considered the same problem as Migorski but with $\rho^\varepsilon=\varepsilon$ and a non-linear source term. Yang and Zhao~\cite{yz} addressed the same problem as Migorski by means of the periodic unfolding method. It is worth pointing out that none of the just mentioned works falls within the framework of reiterated homogenization and those in perforated domains deal with holes of size $\varepsilon$.

In the situations where $\rho\neq 0$ and $\beta=0$, or $\rho=0$ and $\beta\neq 0$ there are numerous works that are indeed related to the homogenization problem for H-P equations. In this direction we quote~\cite{CJ, DN, DY, WD, DW15,   gaveau,   nabil, NR02, NR21, CPAA} and references therein. We also mention that Nnang~\cite{nnang} has studied the deterministic homogenization problem for weakly damped nonlinear H-P equations in a fixed domain with $\rho=1$.

 In this work, the matrix $A$ oscillates on two scale and our domain is perforated with tiny holes of size $\varepsilon^2$ so that our work falls within the scope of reiterated homogenization. Moreover, we have a time dependent function $\beta^\varepsilon$ and we utilised Nguetseng's two scale convergence method~\cite{gabi}. A passage to the limit (as $\varepsilon\to 0$) yields a macroscopic problem which is of the same type as the $\varepsilon$-problem: an H-P equation.

The paper is organized as follows. Section\,\ref{S11} deals with the geometric setting of the problem and detailed assumptions on the data.  In Section\,\ref{S2} some estimates and compactness results are proven. In Section\,\ref{S3}, we recall the basics of the multi-scale convergence theory and formulate a suitable version of its main compactness theorem to be used in the proof of our main result. We also proved some  preliminary convergence results. In the fourth section our main result is formulated and proved.

\section{Setting of the problem}\label{S11}

Let us recall here the setting for the perforated domain  $\Omega^\varepsilon$ (see e.g., the pioneering work on homogenization of differential equations  in perforated domains \cite{CJ}). Let $Z=(0,1)^{N}$ be the unit cube in $\mathbb{R}^{N}$ and let $\Theta\subset Z$ be a compact set in $\mathbb{R}^N$ with a smooth boundary $\partial \Theta$, a non-empty interior and such that the Lebesgue measure of the set $Z\setminus \Theta$ is different from zero. For $\varepsilon >0$, we set
\begin{equation*}
t^\varepsilon =\{k\in \mathbb{Z} : \varepsilon^2(k+\Theta)\subset \Omega\},
\end{equation*}%
\begin{equation*}
\Theta^\varepsilon=\bigcup_{k\in t^\varepsilon} \varepsilon^2(k+\Theta)
\end{equation*}%
and we define the porous medium as:
\begin{equation*}
\Omega ^{\varepsilon }=\Omega \setminus \Theta^\varepsilon.
\end{equation*}%
It appears by construction that $t^\varepsilon$ is finite since $\Omega$ is bounded. Hence $\Theta^\varepsilon$ is closed and $\Omega^\varepsilon$ is open.
One can observe that $\Omega^\varepsilon$ represents the subregion of $\Omega$ obtained from $\Omega$ by removing a finite number of periodically distributed holes
$\{\varepsilon^2 \left(k+\Theta\right) : k\in t^\varepsilon\}$ of size $\varepsilon^2$. In the $\varepsilon$-problem (\ref{eq1}), $\nu_\varepsilon$ is the outward unit normal to $\Omega^\varepsilon$ on $\partial
\Omega^{\varepsilon}\setminus \partial \Omega$. If we set $Z^* = Z \setminus \Theta$ and denote by $\chi_G$ the characteristic function of the set $G$, the perforated domain $\Omega^\varepsilon$ can also be defined by
\begin{equation*}
\Omega^\varepsilon =\Big\{ x \in \Omega\  :\ \chi_{Z^*}\left(  \frac{x}{\varepsilon^2}\right) =1\Big\}.
\end{equation*}
Hence
\begin{equation*}
\chi_{\Omega^\varepsilon} (x)= \chi_{Z^*}\left(  \frac{x}{\varepsilon^2}\right)\qquad (x\in \Omega).
\end{equation*}
For further needs we introduce the following Hilbert space
\begin{equation*}
V_{\varepsilon }=\{u\in H^{1}(\Omega ^{\varepsilon }):u=0\text{ on }\partial \Omega \}
\end{equation*}
endowed with the gradient norm
\begin{equation*}
\| u\| _{V_{\varepsilon }} =\| \nabla u\| _{\left(L^{2}(\Omega
^{\varepsilon })\right)^N}\quad (u\in V_{\varepsilon }).
\end{equation*}

We now state the assumptions on the data. The $\varepsilon $-problem (\ref{eq1}) is constrained as
follows:

\begin{itemize}
\item[\textbf{A1}] \textbf{Uniform ellipticity}. The matrix $A(y,z
)=(a_{ij}(y,z ))_{1\leq i,j\leq N}\in L^{\infty }(\mathbb{R}%
^{N}\times \mathbb{R}%
^{N})^{N\times N}$ is real, symmetric and there exists a positive constant $\Lambda >0$ such that
\begin{eqnarray*}
\begin{array}{l}
\left\Vert a_{ij}\right\Vert _{L^{\infty }(\mathbb{R}^{N}\times \mathbb{R}%
^{N})}\leq \Lambda\ \ \ \text{for all } \ \ 
\,\,1\leq i,j\leq N, \\
\sum\limits_{ij=1}^{N}a_{ij}(y,z )\zeta _{i}\zeta _{j}\geq \Lambda
^{-1}\left\vert \zeta \right\vert ^{2}\ \ \text{ for a.e. }(y,z )\in \mathbb{R}%
^{N}\times \mathbb{R}%
^{N}\ \text{  and all } \ \zeta \in \mathbb{R}^{N}.%
\end{array}%
\end{eqnarray*}

\item[\textbf{A2}] \textbf{Positivity of $\rho$ and $\beta$}.  The functions  $\rho(z)\in \mathcal{C}^1(\mathbb{R}^N)$ and $\beta(y,\tau)\in L^\infty(\mathbb{R}^{N}\times\mathbb{R})$ satisfy

\begin{equation}
\left\{ \begin{aligned} & \rho(z)\geq 0 \quad\text{ in } \ \mathbb{R}^N, \\ & \beta(y,\tau)\geq \alpha >0 \quad\text{a.e. in }\ \mathbb{R}^N\times\mathbb{R}. \end{aligned}\right.
\end{equation}%

\item[\textbf{A3}] \textbf{Periodicity}. Let $Y=(0,1)^N,\ Z=(0,1)^N$ and $\mathcal{T}=(0,1)$. We assume that the function  $\beta$ is $Y\times \mathcal{T}$-periodic and that for any $1\leq i,j\leq N$, the function $a_{ij}$ is $Y\times Z$-periodic. We also  assume that the function  $\rho$ is $Z$-periodic and further satisfy
$$\mathcal{M}_{Z^*}(\rho)=\int_{Z^*}\rho(z)dz>0.$$

\end{itemize}

 The main result of the paper states as follows (the matrix $\widehat{A}$ appearing therein is defined later).

\begin{theorem}
Assume that the hypotheses \textbf{A1-A3} hold and let $u_\varepsilon$ $(\varepsilon>0)$ be the unique solution to (\ref{eq1}). Then as $\varepsilon\to 0$ we have
$$
u_\varepsilon \to u_0 \qquad \text{ in }\quad L^2(\Omega_T),
$$
where $u_0\in L^2(0,T; H^1_0(\Omega))$ is the unique solution to
\begin{equation*}
\left\{
\begin{aligned}
\left(\int_{Z^*}\rho(z)dz\right)
\frac{\partial^2 u_0}{\partial t^2} +
\left( \int_0^1\!\!\!\int_{Y}\beta(y,\tau)dyd\tau \right)
\frac{\partial u_0}{\partial t}
-\frac{1}{|Z^*|}\operatorname{div}\left(\hat{A}\,\nabla_x u_0 \right)
 &= f(x,t) \quad \text{in }\  \Omega\times (0,T),\\
u_{0}&=  \ 0  \ \    \qquad \text{on }\  \partial\Omega\times (0,T),\\
 u_{0}(x,0) & = \ u^{0}(x)\qquad \text{in }\ \ \Omega,\\
 \left(\int_{Z^*}\rho(z)dz\right)\frac{\partial u_0}{\partial t}(x,0) &=\!\! \left(\int_{Z^*}\!\!\!\sqrt{\rho(z)}dz\!\right)\!v^0(x)\ \text{in }\  \Omega.
 \end{aligned}
\right.
\end{equation*}
\end{theorem}

Unless otherwise specified, vector spaces throughout are considered over $\mathbb{R}$, and scalar functions are assumed to take real values. The numerical space $\mathbb{R}^N$ and its open sets are provided with the Lebesgue measure denoted by $dx=dx_1...dx_N$. The usual gradient operator will be denoted by $\nabla$. Throughout, $C$ denotes a generic constant independent of $\varepsilon $ that can change from one line to the next. We will make use of the following notations. The centered dot stands for the Euclidean scalar product in $\mathbb{R}^{N}$ while the absolute value or modulus is denoted by $|\cdot |$. 

Let
$F(\mathbb{R}^m),\ (m\geq 3\ \text{integer})$ be a given function space and let $U$ be a bounded domain in $\mathbb{R}^m$. The Lebesgue measure of $U$ is denoted by $|U|$ and the mean value of a function $v$ over $U$ is denoted and defined by
$$\mathcal{M}_U(v)=\frac{1}{|U|}\int_U v(x)\, dx.$$
We denote by $F_{per}(U)$ the space of functions in $F_{loc}(\mathbb{R}^m)$ (when it makes sense) that are $U$-periodic, and by
$F_{\#} (U)$ the space of those functions $v\in
F_{per}(U)$ with $\int_U v(y)dy=0$. 

The letter $E$ denotes throughout a family of  strictly positive real numbers
$(0<\varepsilon<1)$ admitting $0$ as accumulation point while a  fundamental sequence is any ordinary
sequence of real numbers $0 < \varepsilon_n < 1$, such that $\varepsilon_n\to 0$ as $n\to +\infty$. The time derivatives $\frac{\partial u}{\partial t}$  and $\frac{\partial^2 u}{\partial t^2}$ are sometimes denoted by $u'$ and $u''$, respectively. For $\varepsilon>0$ the functions $x\mapsto\chi_{Z^*}(\frac{x}{\varepsilon^2})$, $x\mapsto\rho(\frac{x}{\varepsilon^2})$, $(x,t)\mapsto\beta(\frac{x}{\varepsilon},\frac{t}{\varepsilon^2})$ and $x\mapsto A(\frac{x}{\varepsilon},\frac{x}{\varepsilon^2})$ are sometimes denoted by $\chi_{Z^*}^\varepsilon$,  $\rho^\varepsilon$, $\beta^\varepsilon$ and $A^\varepsilon$, respectively.

\section{Estimates and compactness results\label{S2}}

We recall that (\cite[Theorem 1.1]{blp}) for any $\varepsilon>0$ the evolution problem \eqref{eq1} admits a unique solution $u_\varepsilon$ that satisfies 
\begin{align*}
&u_\varepsilon \in L^\infty(0,T; V_\varepsilon) \cap L^2(0,T; V'_\varepsilon), \\&
u'_\varepsilon \in L^2(0, T; L^2(\Omega^\varepsilon)), \  \  \ \
\sqrt{\rho^\varepsilon}u'_\varepsilon \in L^\infty(0,T;  L^2(\Omega^\varepsilon)),\\&
\rho^\varepsilon u''_\varepsilon \in L^2(0,T; V'_\varepsilon) \ \ \text{and }\  \\&
u_\varepsilon(0)= u^0 ,\ \ \   \rho^\varepsilon u'_{\varepsilon}(0)= \sqrt{\rho^\varepsilon}v^0.
\end{align*}

\begin{proposition}\label{prop2}
Under the hypotheses {\text (A1)--(A3)}, the following estimates hold:
\begin{gather}  \label{est1}
\|u_\varepsilon\|_{L^\infty(0,T; V_\varepsilon)}\leq C, \\
 \label{est2}
 \|\sqrt{\rho^\varepsilon}u'_\varepsilon \|_{L^\infty(0,T;  L^2(\Omega^\varepsilon))} \leq C, \\
\label{est3}
\|u'_\varepsilon\|_{L^2(0, T; L^2(\Omega^\varepsilon))}\leq C,\\
\label{est4}
\|\rho^\varepsilon u''_\varepsilon\|_{L^2(0,T; V'_\varepsilon)}\leq C,
\end{gather}
where $C$ is a positive constant which does not depend on
$\varepsilon$.
\end{proposition}
\begin{proof}
We follow \cite{blp}. Let $t\in [0,T]$. We multiply the first equation of (\ref{eq1}) by $u'_\varepsilon$ and integrate over $\Omega^\varepsilon$ to get
\begin{equation*}
\int_{\Omega^\varepsilon} \left[\rho\left( \frac{x}{\varepsilon^2}\right)u''_\varepsilon u'_\varepsilon
+ \beta\left( \frac{x}{\varepsilon}, \frac{t}{\varepsilon^2}\right)(u'_\varepsilon)^2
- u'_\varepsilon \operatorname{div}\left( A\left(\frac{x}{\varepsilon
}, \frac{x}{\varepsilon ^{2}}\right) \nabla u_{\varepsilon }  \right) \right]dx = \int_{\Omega^\varepsilon} fu'_\varepsilon dx.
\end{equation*}
Which also writes
\begin{equation}\label{eq2.55}
 \begin{gathered}
\frac{1}{2}\int_{\Omega^\varepsilon}\rho\left( \frac{x}{\varepsilon^2}\right)\left[ (u'_\varepsilon)^2\right]'dx
+ \int_{\Omega^\varepsilon}  \beta\left( \frac{x}{\varepsilon}, \frac{t}{\varepsilon^2}\right)(u'_\varepsilon)^2 dx
 -  \int_{\Omega^\varepsilon} u'_\varepsilon  \operatorname{div}\left( A\left(\frac{x}{\varepsilon
}, \frac{x}{\varepsilon ^{2}}\right) \nabla u_{\varepsilon }  \right)dx =  \int_{\Omega^\varepsilon} fu'_\varepsilon dx.
 \end{gathered}
\end{equation}
But
\begin{align*}
 \begin{gathered}
\frac{1}{2}\int_{\Omega^\varepsilon}\rho\left( \frac{x}{\varepsilon^2}\right)\left[ (u'_\varepsilon)^2\right]'dx =
\frac{1}{2}\frac{d}{dt}\left(\int_{\Omega^\varepsilon}\rho\left( \frac{x}{\varepsilon^2}\right)\left[ (u'_\varepsilon)^2\right]dx \right)= \frac{1}{2}\frac{d}{dt}\left(\rho\left( \frac{x}{\varepsilon^2}\right)u'_\varepsilon, u'_\varepsilon \right)_{L^2(\Omega^\varepsilon)}
\end{gathered}
\end{align*}
and
\begin{equation*}
\int_{\Omega^\varepsilon}  \beta\left( \frac{x}{\varepsilon}, \frac{t}{\varepsilon^2}\right)(u'_\varepsilon)^2 dx
= \left(\beta^\varepsilon u'_\varepsilon, u'_\varepsilon \right)_{L^2(\Omega^\varepsilon)},
\end{equation*}
so that, on setting $\int_{\Omega^\varepsilon}A^\varepsilon\nabla u_\varepsilon \nabla u'_\varepsilon dx
 :=\mathcal{A}^\varepsilon\left( u_\varepsilon, u'_\varepsilon \right)$, the Green formula
\begin{equation*}
 -  \int_{\Omega^\varepsilon} u'_\varepsilon  \operatorname{div}\left( A\left(\frac{x}{\varepsilon
}, \frac{x}{\varepsilon ^{2}}\right) \nabla u_{\varepsilon }  \right)dx = \int_{\Omega^\varepsilon}A^\varepsilon\nabla u_\varepsilon \nabla u'_\varepsilon dx
 \end{equation*}
 and the following consequence of the symmetry  hypothesis on $A$
 \begin{equation*}
 -  \int_{\Omega^\varepsilon} u'_\varepsilon  \operatorname{div}\left( A\left(\frac{x}{\varepsilon
}, \frac{x}{\varepsilon ^{2}}\right) \nabla u_{\varepsilon }  \right)dx =
\frac{1}{2}\frac{d}{dt} \mathcal{A}^\varepsilon\left( u_\varepsilon, u_\varepsilon \right)
\end{equation*}
allow us to rewrite (\ref{eq2.55}) as follows
 \begin{equation}\label{forest1}
 \frac{1}{2}\frac{d}{dt}\left(\rho^\varepsilon u'_\varepsilon, u'_\varepsilon \right)_{L^2(\Omega^\varepsilon)} +
  \left(\beta^\varepsilon u'_\varepsilon, u'_\varepsilon \right)_{L^2(\Omega^\varepsilon)}+
  \frac{1}{2}\frac{d}{dt} \mathcal{A}^\varepsilon\left( u_\varepsilon, u_\varepsilon \right) = (f,u'_\varepsilon)_{L^2(\Omega^\varepsilon)}.
 \end{equation}
We now integrate (\ref{forest1}) on $[0,t]$ and obtain
\begin{equation*}
\begin{gathered}
\frac{1}{2}\left(\rho^\varepsilon u'_\varepsilon (t), u'_\varepsilon (t)\right)_{L^2(\Omega^\varepsilon)} -\frac{1}{2}\left(\rho^\varepsilon u'_\varepsilon (0), u'_\varepsilon (0) \right)_{L^2(\Omega^\varepsilon)}
+\frac{1}{2} \mathcal{A}^\varepsilon\left( u_\varepsilon (t), u_\varepsilon (t) \right) - \frac{1}{2} \mathcal{A}^\varepsilon\left( u_\varepsilon (0), u_\varepsilon (0) \right)  \\ + \int_0^t  \left(\beta^\varepsilon u'_\varepsilon(s), u'_\varepsilon(s) \right)_{L^2(\Omega^\varepsilon)}ds = \int_0^t (f(s),u'_\varepsilon (s))_{L^2(\Omega^\varepsilon)}ds.
\end{gathered}
\end{equation*}
Making use of the initial conditions, we get
\begin{equation*}
\begin{gathered}
\frac{1}{2}\| \sqrt{\rho^\varepsilon} u'_\varepsilon(t)\|^2_{L^2(\Omega^\varepsilon)} + \frac{1}{2} \mathcal{A}^\varepsilon\left( u_\varepsilon (t), u_\varepsilon (t) \right) + \int_0^t  \left(\beta^\varepsilon u'_\varepsilon(s), u'_\varepsilon(s) \right)_{L^2(\Omega^\varepsilon)}ds\\ =
\frac{1}{2} \mathcal{A}^\varepsilon\left( u^0, u^0 \right) + \frac{1}{2}\| v^0\|_{L^2(\Omega^\varepsilon)}^2 + \int_0^t (f(s),u'_\varepsilon (s))_{L^2(\Omega^\varepsilon)}ds.
\end{gathered}
\end{equation*}
Using the positivity of $\beta$, the boundedness and ellipticity hypotheses on $A$, the Cauchy-Schwartz and Young's inequalities, one readily arrives at
\begin{equation*}
\begin{gathered}
\| \sqrt{\rho^\varepsilon} u'_\varepsilon(t)\|^2_{L^2(\Omega^\varepsilon)} +
\frac{1}{\Lambda}\|  u_\varepsilon(t)\|^2_{V_\varepsilon} +
 2\alpha \int_0^t \|  u'_\varepsilon(s)\|^2_{L^2(\Omega^\varepsilon)}ds \\
 \leq  \Lambda \| \nabla u^0 \|^2_{L^2(\Omega)^N} + \| v^0 \|^2_{L^2(\Omega)} +
 \alpha  \int_0^t \|  u'_\varepsilon(s)\|^2_{L^2(\Omega^\varepsilon)}ds +
 \frac{2}{\alpha}\int_0^t \| f(s)\|^2_{L^2(\Omega)}ds,
\end{gathered}
\end{equation*}
which implies (\ref{est1})-(\ref{est3}) as easily seen from
\begin{equation}\label{est}
\begin{gathered}
\| \sqrt{\rho^\varepsilon} u'_\varepsilon(t)\|^2_{L^2(\Omega^\varepsilon)} +
\|  u_\varepsilon(t)\|^2_{V_\varepsilon} +
 \int_0^t \|  u'_\varepsilon(s)\|^2_{L^2(\Omega^\varepsilon)}ds \\
 \leq C\left(  \|  u^0 \|^2_{H^1_0(\Omega)} + \| v^0 \|^2_{L^2(\Omega)} + \| f \|^2_{L^2(0,T;L^2(\Omega))} \right).
 \end{gathered}
\end{equation}
We use the main equation in (\ref{eq1}) to deduce that
$$
\| \rho^\varepsilon u''_\varepsilon \|_{L^2(0,T; V'_\varepsilon)} = \| -\beta^\varepsilon u'_\varepsilon +
 \operatorname{div}\left( A^\varepsilon \nabla u_\varepsilon \right) + f  \|_{L^2(0,T; V'_\varepsilon)}\leq C,
$$
which completes the proof.
\end{proof}

Since solutions of (\ref{eq1}) are defined on $\Omega_T^\varepsilon = (0,T) \times \Omega^\varepsilon$ but not on
$\Omega_T = (0,T) \times \Omega$, we introduce a family of  extension operators so that the sequence of extensions to $\Omega$  of solutions to (\ref{eq1}) lives in a fixed space in which we can study its  asymptotic behaviour. The following result is a classical extension property \cite{CJ, mgk}.

\begin{proposition}\label{prop4}
For any $\varepsilon > 0$, there exists a bounded linear operator  $P_\varepsilon$ such that
$$P_\varepsilon \in \mathcal{L}\left( L^2(0,T; V_\varepsilon) \ ; \ L^2(0,T;H^1_0(\Omega)) \right)
\cap \mathcal{L}\left( L^2(0,T; L^2(\Omega^\varepsilon)) \ ; \  L^2(0,T;L^2(\Omega)) \right)$$
and
\begin{gather}
P_\varepsilon u = u\quad \text{a.e. in }\ \ \Omega^\varepsilon_T,   \\
P_\varepsilon u' = (P_\varepsilon u)'\quad \text{a.e. in }\ \ \Omega^\varepsilon_T,\\
\|P_\varepsilon u\|_{L^2(0,T;L^2(\Omega))}\leq C \|u\|_{L^2(0,T;L^2(\Omega^\varepsilon))},\\
\|P_\varepsilon u\|_{L^2(0,T;H^1_0(\Omega))}\leq C \|u\|_{L^2(0,T;V_\varepsilon)}.
\end{gather}
\end{proposition}
An immediate consequence of Proposition~\ref{prop2} and Proposition~\ref{prop4} is the following estimates that will be useful in the sequel.
\begin{proposition}\label{prop5}
Let $\varepsilon>0$ and let $u_\varepsilon$ be the solution to (\ref{eq1}). There
exists a constant $C>0$ independent of $\varepsilon$ such that
\begin{gather}
\label{eq2.122} \|P_\varepsilon u_\varepsilon \|_{L^2(0,T;H^1_0(\Omega))}\leq C,  \\
\|(P_\varepsilon u_\varepsilon)'\|_{L^2(0,T;L^2(\Omega))} \leq C, \label{eq2.133}\\
\|(P_\varepsilon u_\varepsilon)'\|_{L^2(0,T;H^{-1}(\Omega))}  \leq C.\label{eq2.144}
\end{gather}
\end{proposition}
We are now in position to formulate our first compactness result.
\begin{theorem}\label{theo1}
The sequence $(P_\varepsilon u_\varepsilon)_{\varepsilon > 0}$ is relatively compact in $L^2(0,T;L^2(\Omega))$.
\end{theorem}
\begin{proof}
It is a consequence of proposition \ref{prop5} and a classical embedding result. We define
$$W= \{u\in L^2(0,T;H^1_0(\Omega)) : \  u'\in L^2(0,T;H^{-1}(\Omega)) \}$$
and endow it  with the norm
$$\| u \|_W = \| u \|_ {L^2(0,T;H^1_0(\Omega))} + \| u \|_{L^2(0,T;H^{-1}(\Omega))}\quad (u\in W). $$
It is well known from the Aubin-Lions' lemma that the injection $W\subset \subset L^2(0,T; L^2(\Omega))$  is compact. The proof is completed since the sequence $(P_\varepsilon u_\varepsilon)_{\varepsilon>0}$ is bounded in $W$ as seen from (\ref{eq2.122})-(\ref{eq2.144}).
\end{proof}

\section{Multiscale convergence and preliminary results\label{S3}}
In this section, we recall the definition and main compactness theorem of the multi-scale convergence theory \cite{AB96, gabi}. We also adapt some existing results in this method to our framework. We eventually prove
some preliminary convergence results needed in the homogenization
process of problem (\ref{eq1}).

\subsection{Multiscale convergence method}
\begin{definition} \label{d1}
\textbf{$\bullet$} A sequence $(u_\varepsilon)_{\varepsilon\in E}\subset
L^2(\Omega_T)$ is said to weakly multi-scale converge towards
$u_0\in L^2(\Omega_T\times Y\times Z\times \mathcal{T})$ (denoted $u_\varepsilon\xrightarrow{w-ms} u_0$) in $L^2(\Omega_T)$, if
as $ \varepsilon\to 0$,
\begin{equation}  \label{eqwms1}
\begin{aligned}
&\int_{\Omega_T}u_\varepsilon(x,t)\varphi\left(x,t,\frac{x}{\varepsilon},
\frac{x}{\varepsilon^2},\frac{t}{\varepsilon^2}\right)\,dx\,dt
&\to\iiiint_{\Omega_T\times Y\times Z\times
\mathcal{T}}u_0(x,t,y,z,\tau)\varphi(x,t,y,z,\tau)
\,dx\,dt\,dy\,dz\,d\tau
\end{aligned}
\end{equation}
for all $\varphi\in L^2(\Omega_T;\mathcal{C}_{\text per}(Y\times
Z\times \mathcal{T }))$.

\textbf{$\bullet$} A sequence $(u_\varepsilon)_{\varepsilon\in E}\subset
L^2(\Omega_T)$ is said to strongly multi-scale converge towards
$u_0\in L^2(\Omega_T\times Y\times Z\times \mathcal{T})$ (denoted $u_\varepsilon\xrightarrow{s-ms} u_0$) in $L^2(\Omega_T)$, if it weakly  multi-scale converges to $u_0$ in $ L^2(\Omega_T\times
Y\times Z\times \mathcal{T})$ and further satisfies
\begin{equation*} 
\|u_{\varepsilon}\|_{L^2(\Omega_T)} \to \|u_0\|_{L^2(\Omega_T\times
Y\times Z\times\mathcal{T})} \quad \text{as } \varepsilon \to 0.
\end{equation*}
\end{definition}
\begin{remark}\label{r2}\  \begin{itemize}
 \item[(i)]Let $u\in L^{2}(\Omega _{T};\mathcal{C}_{\text per}(Y\times
Z\times \mathcal{T}))$ and define for $\varepsilon \in E$ \ \
$u^{\varepsilon }:\Omega _{T}\to
\mathbb{R}$ by
$$u^{\varepsilon }(x,t)=u\left(x,t,\frac{x}{ \varepsilon
},\frac{x}{\varepsilon ^{2}},\frac{t}{\varepsilon ^{2}}\right)
\ \ \ \ \ \ \text{for}\ \ \ \ \  (x,t)\in \Omega _{T}.$$
Then $ u^{\varepsilon }\xrightarrow{w-ms}u$ and $u^{\varepsilon }
\xrightarrow{s-ms}u$ in $L^{2}(\Omega _{T})$ as $ \varepsilon \to 0$.
We also have $u^{\varepsilon }\to \widetilde{u}$ in $L^{2}(\Omega
_{T})$ -weak as $\varepsilon \to 0$, with
\[
\widetilde{u}(x,t)=\iiint_{Y\times Z\times \mathcal{T}}u({\cdot
,\cdot ,y,z,\tau })\,dy\,dz\,d\tau .
\]
\item[(ii)] Let $u\in \mathcal{C}(\overline{\Omega } _{T};L_{\text
per}^{\infty }(Y\times Z\times \mathcal{T}))$ and define $
u^{\varepsilon }$ like in (i) above. Then $u^{\varepsilon }
\xrightarrow{w-ms}u$ in $L^{2}(\Omega _{T})$
 as $\varepsilon \to 0$.
\item[(iii)] If $(u_{\varepsilon })_{\varepsilon \in E}\subset
L^{2}(\Omega _{T})$ and $u_{0}\in L^{2}(\Omega _{T}\times Y\times Z\times
\mathcal{T})$ are such that $u_\varepsilon \xrightarrow{w-ms}u_0$  in $L^{2}(\Omega
_{T})$, then
\eqref{eqwms1} still holds for $\varphi \in
\mathcal{C}(\overline{\Omega } _{T};L_{\text per}^{\infty }(Y\times
Z\times \mathcal{T}))$.

\item[(iv)] Since $\chi_{\Omega^\varepsilon} (x)= \chi_{Z^*}\left(  \frac{x}{\varepsilon^2}\right)$ for almost every $x\in \Omega$ and any $\varepsilon\in E$, we deduce from (ii) above that, as $\varepsilon \to 0$, $\chi_{\Omega^\varepsilon} \xrightarrow{w-ms}\chi_{Z^*}$ in $L^2(\Omega)$.
\end{itemize}
\end{remark}
The following two theorems are the backbone of the multi-scale convergence method \cite{AB96, gabi}.
\begin{theorem}\label{theo3}
Any bounded sequence in $L^2(\Omega_T)$ admits a weakly multi-scale convergent subsequence.
\end{theorem}
\begin{theorem}\label{theo4}
Let ($u_{\varepsilon})_{\varepsilon\in E}$ be a  bounded sequence in
$L^2(0,T;H^1_0(\Omega))$, $E$ being a fundamental sequence. There exist a subsequence still denoted by $(u_{\varepsilon})_{\varepsilon\in E}$ and
a triplet $(u_0,u_1,u_2)$ in the space
\[
L^2(0,T;H^1_0(\Omega))\times L^2(\Omega_T;L^2(\mathcal{T};H^1_{\text
per}(Y)))\times L^2(\Omega_T;L^2(Y\times \mathcal{T};H^1_{\text
per}(Z)))
\]
 such that, as $ \varepsilon\to 0$,
\begin{eqnarray}
 u_\varepsilon & \to & u_0 \qquad \qquad\qquad \qquad\text{in }
L^2(0,T;H^1_0(\Omega))\text{-weak}
\label{eqwms3} \\
 \frac{\partial u_\varepsilon}{\partial x_i} & \xrightarrow{w-ms}&
 \frac{\partial u_0}{ \partial x_i}+\frac{\partial u_1}{\partial y_i}+\frac{
\partial u_2}{\partial z_i} \quad \text{in } L^2(\Omega_T)\quad (1\leq
i\leq N).  \label{eqwms4}
\end{eqnarray}
\end{theorem}
\begin{remark}\label{r3} \text
In theorem \ref{theo4}, the functions $u_1$ and $u_2$ are unique up to additive functions of variables $x,t, \tau$ and
$x, t, y, \tau $, repectively. It is therefore crucial to fix the choice of $u_1$ and $u_2$ in accordance with our needs. To formulate the version of theorem~\ref{theo4} we will use, we introduce the space
$$
H^1_{\#\rho}(Z^*)=\left\{ u\in H^1_{\text per}(Z)\ : \  \int_{Z^*}\rho(z)u(z)\,dz=0\right\}
$$
and its dense subspace
$$
\mathcal{C}^\infty_{\#\rho}(Z^*)= \left\{ u\in\mathcal{C}^\infty_{\text per}(Z)\ :\  \int_{Z^*}\rho(z)u(z)\,dz=0\right\}.
$$

\end{remark}
\begin{theorem}\label{theo5}
Let ($u_{\varepsilon})_{\varepsilon\in E}$ be a bounded sequence in
$L^2(0,T;H^1_0(\Omega))$, $E$ being a fundamental sequence. There exist a subsequence still denoted by $(u_{\varepsilon})_{\varepsilon\in E}$ and
a triplet $(u_0,u_1,u_2)$ in the space
\[
L^2(0,T;H^1_0(\Omega))\times
L^2(\Omega_T;L^2(\mathcal{T};H^1_{\#}(Y)))\times
L^2(\Omega_T;L^2(Y\times \mathcal{T};H^1_{\#\rho}(Z^*)))
\]
 such that, as $\varepsilon\to 0$,
\begin{eqnarray}
u_\varepsilon  &\to&  u_0 \qquad \ \  \text{in }
L^2(0,T;H^1_0(\Omega))\text{-weak}
\label{eqwms5} \\
\frac{\partial u_\varepsilon}{\partial x_i} & \xrightarrow{w-ms}&
 \frac{ \partial u_0}{ \partial x_i}+\frac{\partial u_1}{\partial y_i}+\frac{
\partial u_2}{\partial z_i} \quad \text{in } L^2(\Omega_T)\quad
(1\leq j\leq N).  \label{eqwms6}
\end{eqnarray}
\end{theorem}

\subsection{Preliminary results}
Before formulating some preliminary convergence results needed later, we recall some results on periodic distributions (see e.g., \cite{ WD,woukengaa}). As above, let $L^2_{\#\rho}(Z^*)$ denotes the space of functions $u\in L^2_{\text per}(Z)$  with $\int_{Z^*}\rho(z)u(z)dz=0$, and consider the following Gelfand triple
$$
H^1_{\#\rho}(Z^*)\subset L^2_{\#\rho}(Z^*) \subset \left(H^1_{\#\rho}(Z^*)\right)'.
$$
If $u \in L^2_{\#\rho}(Z^*)$ and $ v \in H^1_{\#\rho}(Z^*)$, we have $
\left[ u, v \right] = (u, v)
$
 where $[\cdot, \cdot]$  denotes the duality pairing between $\left(H^1_{\#\rho}(Z^*)\right)'$ and
$H^1_{\#\rho}(Z^*)$ while $(\cdot,\cdot)$ denotes the scalar product in $ L^2_{\#\rho}(Z^*)$. The topological dual of $ L^2(Y\times \mathcal{T};H^1_{\#\rho}(Z^*)) $
is $ L^2(Y\times \mathcal{T};(H^1_{\#\rho}(Z^*))')$ and $ \mathcal{C}^\infty_{\text per}(Y)\otimes \mathcal{C}^\infty_{\text per}(\mathcal{T})\otimes \mathcal{C}^\infty_{\#\rho}(Z^*)  $ is dense in
$ L^2(Y\times \mathcal{T};H^1_{\#\rho}(Z^*)) $.
\begin{proposition}\label{prop8}
Let $u\in \mathcal{D}_{\text per}'(Y\times\mathcal{T}\times Z )$ and assume
that $u$ is continuous on $ \mathcal{C}^\infty_{\text per}(Y)\otimes \mathcal{C}^\infty_{\text per}(\mathcal{T})\otimes \mathcal{C}^\infty_{\#\rho}(Z^*)$ endowed with the $L_{\text
per}^{2}(Y\times \mathcal{T};H^1_{\#\rho }(Z^*)) $-norm. Then $u\in L_{\text
per}^{2}(Y\times\mathcal{T};(H_{\#\rho }^{1}(Z^*))')$, and further
\begin{equation*}
\langle u,\varphi \rangle =\int_{0}^{1}\!\!\!\int_Y \left[ u(y,\tau ),\varphi (y
,\tau, \cdot ) \right]dyd\tau
\end{equation*}
for all $\varphi \in \mathcal{C}^\infty_{\text per}(Y)\otimes \mathcal{C}^\infty_{\text per}(\mathcal{T})\otimes \mathcal{C}^\infty_{\#\rho}(Z^*)$, where $\langle \cdot ,\cdot
\rangle $ denotes the duality pairing between $\mathcal{D}_{\text
per}'(Y\times \mathcal{T}\times Z )$ and $\mathcal{C}^\infty_{\text
per}(Y\times \mathcal{T}\times Z )$.
\end{proposition}

\begin{proposition}\label{prop10}
Let $$
V = \left\{ u\in   L^2(Y\times \mathcal{T};H^1_{\#\rho}(Z^*)):
 \rho\chi_{Z^*}\frac{\partial^2 u}{\partial \tau^2} \in  L^2(Y\times \mathcal{T};(H^1_{\#\rho}(Z^*))') \right\}.
$$
(i) The space $V$ is a reflexive Banach space when endowed with the norm
 $$
 \|u \| = \|u \|_{ L^2(Y\times \mathcal{T};H^1_{\#\rho}(Z^*))} +
  \left\|\rho\chi_{Z^*} \frac{\partial^2 u}{\partial \tau^2} \right\|_{ L^2(Y\times \mathcal{T};(H^1_{\#\rho}(Z^*))')}\qquad\qquad (u\in V).
 $$
 (ii) It holds that 
$$
\int_0^1\!\!\!\int_Y \left[ \rho\chi_{Z^*}\frac{\partial^2 u}{\partial \tau^2},v \right]dy d\tau =
\int_0^1\!\!\!\int_Y  \left[ u,\rho\chi_{Z^*}\frac{\partial^2 v}{\partial \tau^2} \right]dy d\tau\qquad\text{ for all } \ u,v\in V.
$$
\end{proposition}
We can now formulate the main result of this section.
\begin{theorem}\label{theo6} Let $(u_{\varepsilon})_{\varepsilon\in E}$ be the sequence of solution to (\ref{eq1}), $E$ being a fundamental sequence. There exist a subsequence $E'$ of $E$  and a triplet $(u_0,u_1,u_2)$ in the space
\[
L^2(0,T;H^1_0(\Omega))\times
L^2(\Omega_T; H^1_{\#}(Y))\times
L^2(\Omega_T;L^2(Y\times \mathcal{T};H^1_{\#\rho}(Z^*)))
\]
 such that, as $E'\ni \varepsilon\to 0$,
\begin{eqnarray}
P_\varepsilon u_\varepsilon & \to & u_0 \qquad \ \  \text{in }\ \ 
L^2(\Omega_T),
\label{eqwms55} \\
\frac{\partial (P_\varepsilon u_\varepsilon)}{\partial t} &\xrightarrow{w-ms}& \frac{\partial u_0}{\partial t}\qquad \text{in }\ \  L^2(\Omega_T),\label{eqwms555}\\
 \frac{\partial (P_\varepsilon u_\varepsilon)}{\partial x_i} & \xrightarrow{w-ms} &
 \frac{ \partial u_0}{ \partial x_i}+\frac{\partial u_1}{\partial y_i}+\frac{
\partial u_2}{\partial z_i} \quad \text{in }\ \  L^2(\Omega_T)\quad
(1\leq i\leq N).  \label{eqwms66}
\end{eqnarray}
\end{theorem}
The proof of Theorem~\ref{theo6} requires two preliminary results and is therefore postponed. 

\begin{lemma}\label{key4} Let $E,\ E'$,  $(u_{\varepsilon})_{\varepsilon\in E}$ and the triplet
$(u_0,u_1,u_2)$ be as in Theorem~(\ref{theo6}). It holds that
\begin{eqnarray*}
&&\lim_{ \varepsilon \to 0} \frac{1}{\varepsilon^2 }
\int_{\Omega_{T}}u_{\varepsilon
}(x,t)\rho
\left(\frac{x}{\varepsilon^2 } \right)\chi_{Z^*}\left(\frac{x}{\varepsilon^2}\right)\varphi \left(x,t,\frac{x}{\varepsilon
},\frac{x}{\varepsilon^2 },\frac{t}{\varepsilon ^{2}}\right)dx\,dt\\&&\qquad \qquad
=\iiiint_{\Omega_T\times Y \times Z \times \mathcal{T}}u_2(x, t,
y,z,
\tau)\rho(z)\chi_{Z^*}(z)\varphi(x,t,y,z,\tau)\,dx\,dt\,dy\,dz\,d\tau
\end{eqnarray*}
for all $\varphi \in \mathcal{D}(\Omega_{T})\otimes \mathcal{C}^\infty_{\text{per}}(Y)\otimes \mathcal{C}^\infty_{\text per}(Z)\otimes \mathcal{C}^\infty_{\text per}(\mathcal{T})$ such that
$$ \int_{Z}\chi_{Z^*}\rho(z)\varphi(z)dz = 0\quad  \text{ for all }\  (x,t,y,\tau)\in\Omega_T\times Y\times\mathcal{T}.$$ 
\end{lemma}
\begin{proof}
 Let $\varphi \in \mathcal{D}(\Omega_{T})\otimes \mathcal{C}^\infty_{\text{per}}(Y)\otimes \mathcal{C}^\infty_{\text per}(Z)\otimes \mathcal{C}^\infty_{\text per}(\mathcal{T})$ with 
$
\int_{Z}\chi_{Z^*}\rho(z)\varphi(z)dz = 0,
$
we deduce from the
Fredholm alternative the existence of a unique $w \in
\mathcal{D}(\Omega_{T})\otimes \mathcal{C}^\infty_{\text per}(Y)\otimes H^1_{\#\rho}(Z^*)\otimes \mathcal{C}^\infty_{\text per}(\mathcal{T})$ such that 
\begin{equation}\label{eq3.1111}
\left\{
\begin{aligned}
&\Delta_z w =\varphi \rho\chi_{Z^*} \quad \text{ in } Z,\\&
w(x,t,y,\tau)\in H^1_{\#\rho}(Z^*)\ \ \ \ \text{for all}\ (x,t,y,\tau)\in\Omega_T\times Y \times \mathcal{T}.
\end{aligned}
 \right.
\end{equation}
But the restriction to $Z^*$ of the function $w$ defined by (\ref{eq3.1111}) lives in $\mathcal{C}^3(Z^*)$ so that we have  
\begin{equation*}
\text{div}(\nabla_z w)^\varepsilon = (\text{div}\nabla_z w)^\varepsilon + \frac{1}{\varepsilon}(\text{div}_y\nabla_z w)^\varepsilon+\frac{1}{\varepsilon^2}(\Delta_z w)^\varepsilon\qquad \text{in }\ \Omega^\varepsilon_T,
\end{equation*}
and therefore
\begin{eqnarray*}
 \frac{1}{\varepsilon^2 }
\int_{\Omega_{T}}u_{\varepsilon
}(x,t)\rho
\left(\frac{x}{\varepsilon^2 } \right)\chi_{Z^*}\left(\frac{x}{\varepsilon^2}\right)\varphi \left(x,t,\frac{x}{\varepsilon
},\frac{x}{\varepsilon^2 },\frac{t}{\varepsilon ^{2}}\right)dx\,dt &=&- \int_{\Omega_T}\nabla u_\varepsilon\cdot\chi_{Z^*}^\varepsilon(\nabla_z w)^\varepsilon dxdt \\- \int_{\Omega_T}u_\varepsilon\chi_{Z^*}^\varepsilon(\text{div} \nabla_z w)^\varepsilon dxdt-\frac{1}{\varepsilon}\int_{\Omega_T}u_\varepsilon\chi_{Z^*}^\varepsilon(\text{div}_y\nabla_z w)^\varepsilon dxdt. 
\end{eqnarray*}
As, $E'\ni \varepsilon\to 0$, (\ref{eqwms66}) and (iii) of Remark~\ref{r2} reveal that the first term in the right hand side of this equality converges to 
\begin{eqnarray*}
&& -\iiiint_{\Omega_T \times Y\times Z\times\mathcal{T} }\left(\nabla_x u_0 + \nabla_y u_1 + \nabla_z u_2\right)\cdot \chi_{Z^*}(\nabla_z w)dxdtdydzd\tau\\&&\qquad\qquad\qquad \quad = \iiiint_{\Omega_T \times Y\times Z\times\mathcal{T} } u_2\chi_{Z^*}(\Delta_z w)dxdtdydzd\tau,
 \end{eqnarray*}
while the second one converges to zero. As regards the third term, since the test function therein, $(\text{div}_y\nabla_z w)^\varepsilon$, depends on the $z$ variable, its limit cannot be computed as usual like in \cite[Theorem~2.3]{douanlaaa} even if its mean value over $Y$ is zero. This requires some further investigations. From
\begin{equation*}
\text{div}(\nabla_y w)^\varepsilon = (\text{div}\nabla_y w)^\varepsilon + \frac{1}{\varepsilon}(\text{div}_y\nabla_y w)^\varepsilon+\frac{1}{\varepsilon^2}(\text{div}_z\nabla_y w)^\varepsilon \qquad\text{in }\ \ \Omega^\varepsilon_T
\end{equation*}
and 
\begin{equation*}
(\text{div}_z\nabla_y w)^\varepsilon=(\text{div}_y\nabla_z w)^\varepsilon\qquad\text{in }\ \ \Omega^\varepsilon_T,
\end{equation*}
it follows that
\begin{eqnarray*}
&&-\frac{1}{\varepsilon}\int_{\Omega_T}u_\varepsilon\chi_{Z^*}^\varepsilon(\text{div}_y\nabla_z w)^\varepsilon dxdt =  \varepsilon\int_{\Omega_T}u_\varepsilon \chi_{Z^*}^\varepsilon(\text{div}\nabla_y w)^\varepsilon dxdt \\ && \qquad\qquad\qquad+ \int_{\Omega_T}u_\varepsilon \chi_{Z^*}^\varepsilon(\Delta_y w)^\varepsilon dxdt + \varepsilon\int_{\Omega_T}\nabla u_\varepsilon\cdot\chi_{Z^*}^\varepsilon(\nabla_y w)^\varepsilon dxdt.
\end{eqnarray*}

Therefore
$$-\frac{1}{\varepsilon}\int_{\Omega_T}u_\varepsilon\chi_{Z^*}^\varepsilon(\text{div}_y\nabla_z w)^\varepsilon dxdt \to \iiiint_{\Omega_T\times Y\times Z\times\mathcal{T}}u_0\chi_{Z^*}(\Delta_y w)\, dxdtdydzd\tau=0
$$
as $E'\ni\varepsilon\to 0$. The proof is completed.
\end{proof}

\begin{lemma}\label{key1} Let $E,\ E'$,   $(u_{\varepsilon})_{\varepsilon\in E}$ and the triplet
$(u_0,u_1,u_2)$ be as in Theorem~(\ref{theo6}). It holds that
\begin{eqnarray*}
&&\lim_{ \varepsilon \to 0} \frac{1}{\varepsilon }
\int_{\Omega_{T}}u_{\varepsilon
}(x,t)\rho
\left(\frac{x}{\varepsilon^2 } \right)\chi_{Z^*}\left(\frac{x}{\varepsilon^2}\right)\varphi \left(x,t,\frac{x}{\varepsilon
},\frac{x}{\varepsilon^2 },\frac{t}{\varepsilon ^{2}}\right)dx\,dt\\&&\qquad \qquad
=\iiiint_{\Omega_T\times Y \times Z \times \mathcal{T}}u_1(x, t,
y,z,
\tau)\rho(z)\chi_{Z^*}(z)\varphi(x,t,y,z,\tau)\,dx\,dt\,dy\,dz\,d\tau
\end{eqnarray*}
for all $\varphi \in \mathcal{D}(\Omega_{T})\otimes \mathcal{C}^\infty_{\text per}(Y)\otimes \mathcal{C}^\infty_{\text{per}}(Z)\otimes \mathcal{C}^\infty_{\text per}(\mathcal{T})$ such that
 $$ \int_{Y}\varphi(y)dy = 0\quad  \text{ for all }\  (x,t,z,\tau)\in\Omega_T\times Z\times\mathcal{T}.
 $$ 
\end{lemma}
\begin{proof}
 Let $\varphi \in \mathcal{D}(\Omega_{T})\otimes \mathcal{C}^\infty_{\text per}(Y)\otimes \mathcal{C}^\infty_{\text{per}}(Z)\otimes \mathcal{C}^\infty_{\text per}(\mathcal{T})$ with $\int_{Y}\varphi(y)dy = 0$ and consider $w\in \mathcal{D}(\Omega_{T})\otimes \mathcal{C}^\infty_{\text per}(Y)\otimes \mathcal{C}^1_{\text{per}}(Z)\otimes \mathcal{C}^\infty_{\text per}(\mathcal{T})$ such
that
\begin{equation*}
\left\{
\begin{aligned}
&\Delta_y w = \varphi \rho\quad \text{ in } Y,\\&
w(x,t,z,\tau)\in \mathcal{C}^\infty_{\#}(Y)\ \ \ \text{for all }\ (x,t,z,\tau)\in \Omega_T\times Z\times\mathcal{T}.
\end{aligned}
 \right.
\end{equation*}
Recalling that 
 $$
\text{div}(\nabla_y w)^\varepsilon = (\text{div}\nabla_y w)^\varepsilon + \frac{1}{\varepsilon}(\text{div}_y\nabla_y w)^\varepsilon+\frac{1}{\varepsilon^2}(\text{div}_z\nabla_y w)^\varepsilon \quad \text{in }\ \Omega_T,
$$
the following holds true 
\begin{eqnarray}& &
 \frac{1}{\varepsilon }
\int_{\Omega_{T}}u_{\varepsilon
}(x,t)\rho
\left(\frac{x}{\varepsilon^2 } \right)\chi_{Z^*}\left(\frac{x}{\varepsilon^2}\right)\varphi \left(x,t,\frac{x}{\varepsilon
},\frac{x}{\varepsilon^2 },\frac{t}{\varepsilon ^{2}}\right)dxdt\nonumber\\&&\quad\quad =\frac{1}{\varepsilon }
\int_{\Omega_{T}}u_{\varepsilon
}(x,t)\chi_{Z^*}\left(\frac{x}{\varepsilon^2}\right)(\Delta_y w)\left( x,t,\frac{x}{\varepsilon},\frac{x}{\varepsilon^2},\frac{t}{\varepsilon^2} \right) dxdt\label{eq3.111} \\ & & \quad\quad = -\int_{\Omega_T} \nabla u_\varepsilon\cdot \chi_{Z^*}^\varepsilon (\nabla_y w)^\varepsilon -
\int_{\Omega_T}u_\varepsilon\chi_{Z^*}^\varepsilon (\text{div}\nabla_y w)^\varepsilon dxdt- \frac{1}{\varepsilon^2}\int_{\Omega_T}u_\varepsilon\chi_{Z^*}^\varepsilon (\text{div}_z\nabla_y w)^\varepsilon dxdt.\nonumber
\end{eqnarray}
As $\int_{Z}\text{div}_z(\nabla_y w)dz=0$  we may follow the lines of reasoning in the proof of Lemma~\ref{key4} to compute the limit of the last term in (\ref{eq3.111}). We find that as $E'\ni\varepsilon\to 0$ the right hand side of (\ref{eq3.111}) converges to  
\begin{eqnarray*}
&&\iiiint_{\Omega_T \times Y\times Z\times\mathcal{T} }\chi_{Z^*}\Big[-\left(\nabla_x u_0 + \nabla_y u_1 + \nabla_z u_2\right)\cdot (\nabla_y w) -u_0 (\text{div}\nabla_y w)-u_2(\text{div}_z \nabla_y w)\Big]dxdtdydzd\tau\nonumber \\&&\qquad=\iiiint_{\Omega_T \times Y\times Z\times\mathcal{T} }u_1\chi_{Z^*}(\Delta_y w)\,dxdtdydzd\tau =\iiiint_{\Omega_T \times Y\times Z\times\mathcal{T} }u_1\chi_{Z^*}\varphi\rho  \, dxdtdydzd\tau,
\end{eqnarray*}
and the proof is completed.
\end{proof}

\begin{proof}[\textbf{Proof of Theorem~\ref{theo6}}]
According to Proposition~\ref{prop5}, Theorem~\ref{theo1} and  Theorem~\ref{theo5}, it remains to prove (\ref{eqwms555}) and to justify that the function $u_1$ in the triplet given by Theorem~\ref{theo5} actually lives in $L^2(\Omega_T; H^1_{\#}(Y))$, i.e., $u_1$ does not depend on the variable $\tau$. We start with the fact that $u_1\in L^2(\Omega_T; H^1_{\#}(Y))$. To prove this, let $\psi \in
\mathcal{D}(\Omega_T)\otimes\mathcal{C}^\infty_{\# }(Y)\otimes
\mathcal{C}^\infty_{\text per}(\mathcal{T})$ and consider the function $\psi^\varepsilon \in \mathcal{D}(\Omega_T)$ defined by
$$
\psi^\varepsilon
(x,t) = \varepsilon^3\psi
\left(x,t,\frac{x}{\varepsilon},\frac{t}{\varepsilon^2} \right),\qquad  \ \
\ (x,t)\in\Omega_T.
$$
Using $\psi^\varepsilon $ as a test function in problem (\ref{eq1}), we obtain
\begin{eqnarray*}
&&\int_{\Omega_T}\rho^\varepsilon
\chi_{Z^*}^\varepsilon(P_\varepsilon u_{\varepsilon})\frac{\partial^2\psi^\varepsilon}{\partial
t^2}dxdt +
\int_{\Omega_T}\beta^\varepsilon\chi_{Z^*}^\varepsilon\frac{\partial
(P_\varepsilon u_{\varepsilon})}{\partial t}\psi^\varepsilon dxdt \\ &&\qquad \qquad+
\int_{\Omega_T} A^\varepsilon\nabla(P_\varepsilon u_\varepsilon)\cdot\chi_{Z^*}^\varepsilon\nabla\psi^\varepsilon dxdt = \int_{\Omega_T}
f\chi_{Z^*}^\varepsilon\psi^\varepsilon dxdt.
\end{eqnarray*}
Letting $E'\ni\varepsilon \to 0$ in this equation, the term in the right hand side and the second and third terms on the left hand side obviously converge to zero, so that

\begin{equation}\label{ld22}
\lim_{E'\ni \varepsilon\to 0}\int_{\Omega_T}\rho^\varepsilon
\chi_{\Omega^\varepsilon}(P_\varepsilon u_{\varepsilon})\frac{\partial^2\psi^\varepsilon}{\partial
t^2}dxdt = 0.
\end{equation}
But
\begin{equation}\label{d22}
\frac{\partial^2\psi^\varepsilon}{\partial t^2} =
\varepsilon^3\frac{\partial^2\psi}{\partial t^2}
+2\varepsilon\frac{\partial^2\psi}{\partial t \partial \tau} +
\frac{1}{\varepsilon}\frac{\partial^2\psi}{\partial\tau^2}.
\end{equation}
Plugging (\ref{d22}) into (\ref{ld22}) we realise that 
$$
\lim_{E'\ni\varepsilon \to
0}\frac{1}{\varepsilon}\int_{\Omega_T}(P_\varepsilon u_{\varepsilon})(x,t)\rho\left(\frac{x}{\varepsilon^2}\right)
\chi_{Z^*}\left(\frac{x}{\varepsilon^2}\right)\frac{\partial^2\psi}{\partial
\tau^2}\left(x,t,\frac{x}{\varepsilon},\frac{t}{\varepsilon^2} \right) dxdt = 0.
 $$
 Using Lemma~\ref{key1}, this is equivalent to
 $$
\int\int\int\int_{\Omega_T \times Y \times Z \times
\mathcal{T}}u_1(x,t,
y,\tau)\rho(z)\chi_{Z^*}(z)\frac{\partial^2\psi}{\partial\tau^2}(x,t, y,
\tau)dxdtdydz d\tau =0,
 $$
 which, by taking $\psi = \psi_1\otimes\psi_2\otimes\psi_3$ with
 $\psi_1 \in \mathcal{D}(\Omega_T)\ \ \ \psi_2\in
 \mathcal{C}^\infty_{\#}(Y)$ and $\psi_3 \in \mathcal{C_{\text per}^\infty}(\mathcal{T})$, also writes 
$$
\int_{Z^*}\rho(z)dz\iint_{\Omega_T \times
Y}\psi_1(x,t)\psi_2(y)\left(\int_{\mathcal{T}}u_1(x,t,
y,\tau)\frac{\partial^2\psi_3}{\partial\tau^2}(\tau)d\tau\right)dxdtdy=0.
$$
The hypothesis $\mathcal{M}_{Z^*}(\rho)>0$ and the  arbitrariness of $\psi_1$ and $\psi_2$ yields
$$
\int_0^1 u_1(x,t, y,\tau)\frac{\partial^2 \psi_3}{\partial\tau^2}(\tau)d\tau = 0 \quad \ \text{ for all}\ \ \  \psi_3 \in
\mathcal{C}^\infty_{\text per}(\mathcal{T}).
$$
Taking in particular $\psi_3(\tau)= e^{-2i\pi p \tau}  \ (p \in
\mathbb{Z}\setminus \{0\})$, we obtain
\begin{equation}\label{eq333}
\int_0^1u_1(x,t,y,\tau)e^{-2i\pi p \tau}d\tau =0
\ \ \ \ \  \text{ for all }\  p \in \mathbb{Z}\setminus \{0\}.
\end{equation}
The Fourier series expansion of the periodic function $\tau \mapsto u_1(x, t, y, \tau)$ writes
$$
u_1(x, t, y, \tau) = \sum_{p \in \mathbb{Z}}C_pe^{2i\pi p \tau} \ \
\ \ \ \text{where} \ C_p=\int_0^1u_1(x,t,y,\tau)e^{-2i\pi p
\tau}d\tau.
$$
But (\ref{eq333}) implies that $C_p =0 $ for all $ p \in \mathbb{Z}\setminus \{0\} $, so that
 $ u_1(x,t,y,\tau) = C_0 = \int_0^1 u_1(x,t,y,\tau)d\tau$. This proves that the function $u_1$ is independent of $\tau$.

We now prove (\ref{eqwms555}). It follows from (\ref{eq2.133}) and Theorem~\ref{theo3} that there exists $w\in L^2(\Omega_T\times Y\times Z\times \mathcal{T})$ such that, as $E'\ni\varepsilon\to 0$,
\begin{equation}\label{eq3.122}
\frac{\partial (P_\varepsilon u_\varepsilon)}{\partial t}\xrightarrow{w-ms} w\qquad \text{in} \ \ \ L^2(\Omega_T).
\end{equation}
Since (\ref{eqwms55}) implies that $\frac{\partial (P_\varepsilon u_\varepsilon)}{\partial t}\to \frac{\partial u_0}{\partial t}$ weakly in $\mathcal{D}'(\Omega_T)$ as $E'\ni\varepsilon\to 0$, while (\ref{eq3.122}) implies $\frac{\partial (P_\varepsilon u_\varepsilon)}{\partial t}\xrightarrow{w-ms} \mathcal{M}_{Y\times Z\times\mathcal{T}}(w)$ weakly in $\mathcal{D}'(\Omega_T)$ as $E'\ni\varepsilon\to 0$ it is sufficient to prove that the function $w$ does not depend on the variables $y, z$ and $\tau$ to conclude that $w=\frac{\partial u_0}{\partial t}$ in $L^ 2(\Omega_T)$. Firstly, we prove that the function $w$ does not depend on the variable $z$. Let $\theta\in\mathcal{D}(\Omega_T)$, $\varphi\in\mathcal{C}^\infty_{\text per}(Y)$, $\psi\in \mathcal{C}^\infty_{\text per}(Z)$ and $\vartheta\in \mathcal{C}^\infty_{\text per}(\mathcal{T})$ and define $w^\varepsilon(x,t)=\theta(x,t)\varphi(\frac{x}{\varepsilon})\psi(\frac{x}{\varepsilon^2})\vartheta(\frac{t}{\varepsilon^2})$ for $\varepsilon\in E'$ and $(x,t)\in \Omega_T$. Passing to the limit as $E'\ni\varepsilon\to 0$ in the following equality 

\begin{eqnarray*}
 &&-\varepsilon^2\left\langle \frac{\partial}{\partial x_j}\left(\frac{\partial (P_\varepsilon u_\varepsilon)}{\partial t}\right), w^\varepsilon\right\rangle_{L^2(0,T;H^{-1}(\Omega)),  L^2(0,T;H^1_0(\Omega))}\\ \text{ }\\ && \qquad\qquad=\int_{\Omega_T}\vartheta\left(\frac{t}{\varepsilon^2}\right)\frac{\partial (P_\varepsilon u_\varepsilon)}{\partial t}\left[\varepsilon^2\frac{\partial \theta}{\partial x_j}\varphi^\varepsilon\psi^\varepsilon + \theta\varphi^\varepsilon\left(\frac{\partial \psi}{\partial z_j}\right)^\varepsilon + \varepsilon\theta\left(\frac{\partial \varphi}{\partial y_j}\right)^\varepsilon \psi^\varepsilon\right]dxdt,
\end{eqnarray*}
we obtain (keep in mind that (\ref{eq2.133}) implies the boundedness in $L^2(0,T;H^{-1}(\Omega))$ of the first term in the duality bracket just above)
$$
0=\int_{\Omega_T\times Y\times Z\times \mathcal{T}}w(x,t,y,z,\tau)\theta(x,t)\varphi(y)\frac{\partial \psi}{\partial z_j}(z)\vartheta(\tau)dxdtdydzd\tau
$$
which by the arbitrariness of $\theta, \varphi$ and $\vartheta$ implies that
$$
\int_Z w(x,t,y,z,\tau)\frac{\partial \psi}{\partial z_j}(z)dz=0 \qquad \text{ for all }\  (x,t,y,\tau)\in\Omega_T\times Y\times\mathcal{T},
$$
which proves that $w$ does not depend on the variable $z$. Similarly, one easily proves that $w$ does not depend on $y$ by passing to the limit in the following equality (where $w^\varepsilon(x,t)=\theta(x,t)\varphi(\frac{x}{\varepsilon})\vartheta(\frac{t}{\varepsilon^2})$ for $\varepsilon\in E'$ and $(x,t)\in \Omega_T$)
$$
-\varepsilon\left\langle \frac{\partial}{\partial x_j}\left(\frac{\partial (P_\varepsilon u_\varepsilon)}{\partial t}\right), w^\varepsilon\right\rangle_{L^2(0,T;H^{-1}(\Omega)),  L^2(0,T;H^1_0(\Omega))}=\int_{\Omega_T}\vartheta \left(\frac{t}{\varepsilon^2}\right)\frac{\partial (P_\varepsilon u_\varepsilon)}{\partial t}\left[\varepsilon\frac{\partial \theta}{\partial x_j}\varphi^\varepsilon +\theta\left(\frac{\partial \varphi}{\partial y_j}\right)^\varepsilon \right]dxdt.
$$

As for the independence of $u_1$ from the variable $\tau$, we have the following equality, where $w^\varepsilon(x,t)=\theta(x,t)\vartheta(\frac{t}{\varepsilon^2})$ for $\varepsilon\in E'$ and $(x,t)\in \Omega_T$
$$
-\varepsilon^2\left\langle \rho^\varepsilon\frac{\partial^2  u_\varepsilon}{\partial t^2}, w^\varepsilon\right\rangle_{L^2(0,T;V'_\varepsilon),\, L^2(0,T;V_\varepsilon)}=\int_{\Omega_T}\rho\left(\frac{x}{\varepsilon^2}\right)\chi_{Z^*}\left(\frac{x}{\varepsilon^2}\right)\frac{\partial (P_\varepsilon u_\varepsilon)}{\partial t}\left[\varepsilon^2\frac{\partial \theta}{\partial t}\vartheta^\varepsilon  +\theta\left(\frac{\partial \vartheta}{\partial \tau}\right)^\varepsilon \right]dxdt,
$$
which, after a limit passage as $E'\ni\varepsilon\to 0$ (keeping (\ref{est4}) in mind) leads to
$$0= \mathcal{M}_{Z^*}(\rho)\int_0^1w(x,t,\tau)\frac{\partial \vartheta}{\partial\tau}d\tau \qquad \text{ for all }\  (x,t)\in\Omega_T.$$
But $\mathcal{M}_{Z^*}(\rho)>0$ and the proof is completed.
\end{proof}

\begin{remark}\label{r3.3}
In order to capture all the microscopic and mesoscopic behaviours of the phenomenon modelled by problem (\ref{eq1}), one must take  test functions of the form
\begin{equation*}
\psi_\varepsilon(x,t)=\psi_0(x,t)+\varepsilon\psi_1\left(x,t,\frac{x}{\varepsilon},\frac{t}{\varepsilon^2}\right)+\varepsilon^2\psi_2\left(x,t,\frac{x}{\varepsilon},
\frac{x}{\varepsilon^2},\frac{t}{\varepsilon^2}\right),
\end{equation*}
with  $\psi_0\in\mathcal{D}(\Omega_T)$, $\psi_1\in\mathcal{D}(\Omega_T)\otimes\mathcal{C}^\infty_{\#}(Y)\otimes\mathcal{C}^\infty_{\text per}(\mathcal{T})$ and $\psi_2\in \mathcal{D}(\Omega_T)\otimes\mathcal{C}^\infty_{per}(Y)\otimes \mathcal{C}^\infty_{\#\rho}(Z^*)\otimes\mathcal{C}^\infty_{\text per}(\mathcal{T})$. The  Theorem~\ref{theo6} informs us that the function $u_1$ does not depend on the variable $\tau$ so that in the homogenization process of problem~(\ref{eq1}) we can instead use test functions of the form
\begin{equation}\label{eq3.99}
\psi_\varepsilon(x,t)=\psi_0(x,t)+\varepsilon\psi_1\left(x,t,\frac{x}{\varepsilon}\right)+\varepsilon^2\psi_2\left(x,t,\frac{x}{\varepsilon},
\frac{x}{\varepsilon^2},\frac{t}{\varepsilon^2}\right),
\end{equation}
where $\psi_1\in\mathcal{D}(\Omega_T)\otimes\mathcal{C}^\infty_{\#}(Y)$ and $\psi_0, \psi_2$ are as above.
\end{remark}

\section{The homogenization process}

In this section, we pass to the limit  in the limit in the variational formulation of problem (\ref{eq1}) and formulate the microscopic problem, the mesoscopic problem and the macroscopic problem, successively.

\subsection{The global limit problem for (\ref{eq1})} The setting being that of Theorem~\ref{theo6}, let $\psi_0\in\mathcal{D}(\Omega_T)$, $\psi_1\in\mathcal{D}(\Omega_T)\otimes\mathcal{C}^\infty_{\#}(Y)$ and $\psi_2\in \mathcal{D}(\Omega_T)\otimes\mathcal{C}^\infty_{per}(Y)\otimes \mathcal{C}^\infty_{\#\rho}(Z^*)\otimes\mathcal{C}^\infty_{\text per}(\mathcal{T})$, and consider for any $\varepsilon\in E$, the function $\psi_\varepsilon\in\mathcal{D}(\Omega_T)$ defined as in (\ref{eq3.99}). We aim at passing to the limit (as $E'\ni\varepsilon\to 0$) in the following equality.
\begin{equation}\label{vareq}
\begin{aligned}
&\int_{\Omega_T}\rho\left(
\frac{x}{\varepsilon^2}\right)\chi_{Z^*}\left(
\frac{x}{\varepsilon^2}\right)(P_\varepsilon u_\varepsilon)\frac{\partial^2\psi_\varepsilon}{\partial
t^2}dx\,dt
 +\int_{\Omega_T}\beta\left( \frac{x}{\varepsilon},
\frac{t}{\varepsilon^2} \right)\chi_{Z^*}\left(
\frac{x}{\varepsilon^2}\right)\frac{\partial (P_\varepsilon u_\varepsilon)}{\partial
t}\psi_{\varepsilon}(x,t)\,dx\,dt\\
&+ \int_{\Omega_T}\chi_{Z^*}\left(
\frac{x}{\varepsilon^2}\right)A\left(\frac{x}{\varepsilon},\frac{x}{\varepsilon^2}\right)\nabla
\left(P_\varepsilon u_\varepsilon \right)\cdot\nabla\psi_\varepsilon\,dx\,dt=\int_{\Omega_T}f(x,t)\psi_\varepsilon (x,t)\chi_{Z^*}\left(
\frac{x}{\varepsilon^2}\right)dxdt.
\end{aligned}
\end{equation}
We will consider each term of (\ref{vareq}) separately. We
start with the first term in the left hand side and denote it by $L_1$. Recalling that
$$
\frac{\partial^2\psi_\varepsilon }{\partial t^2}=
\frac{\partial^2\psi_0}{\partial t^2} + \varepsilon
\left(\frac{\partial^2\psi_1 }{\partial t^2}\right)^\varepsilon +
\varepsilon^2\left(\frac{\partial^2\psi_2}{\partial t^2}\right)^\varepsilon+ 2
\left(\frac{\partial^2\psi_2}{\partial t \partial \tau}\right)^\varepsilon +
\frac{1}{\varepsilon^2} \left(\frac{\partial^2\psi_2}{\partial\tau^2}\right)^\varepsilon\qquad\text{in}\ \ \ \Omega_T,
$$
we have
\begin{align}
 L_1 & = \int_{\Omega_T}\rho^\varepsilon\chi_{Z^*}^\varepsilon(P_\varepsilon u_\varepsilon)\frac{\partial^2\psi_0}{\partial
t^2}dx\,dt +  \varepsilon \int_{\Omega_T}\rho^\varepsilon\chi_{Z^*}^\varepsilon(P_\varepsilon u_\varepsilon)\left(\frac{\partial^2\psi_1 }{\partial t^2}\right)^\varepsilon dx\,dt\nonumber \\
&\qquad  + \varepsilon^2 \int_{\Omega_T}\rho^\varepsilon\chi_{Z^*}^\varepsilon(P_\varepsilon u_\varepsilon) \left(\frac{\partial^2\psi_2}{\partial t^2}\right)^\varepsilon  dxdt +2 \int_{\Omega_T}\rho^\varepsilon\chi_{Z^*}^\varepsilon(P_\varepsilon u_\varepsilon) \left(\frac{\partial^2\psi_2}{\partial t \partial \tau}\right)^\varepsilon dxdt  \label{eq4.2}\\& \qquad  +  \frac{1}{\varepsilon^2}
\int_{\Omega_T}\rho^\varepsilon\chi_{Z^*}^\varepsilon(P_\varepsilon u_\varepsilon)(x,t) \left(\frac{\partial^2\psi_2}{\partial\tau^2}\right)^\varepsilon dxdt.\nonumber
\end{align}
As $\varepsilon\to 0$, (\ref{eqwms55}) and (iii) of Remark~\ref{r2}, imply that the second and third terms in the right hand side of (\ref{eq4.2})  converge to zero and the fourth term converges to
$$2\iiint_{\Omega_T\times Y\times Z}
\rho(z)\chi_{Z^*}(z)u_0(x,t)\left(\int_0^1\frac{\partial}{\partial\tau}\left(
\frac{\partial\psi_2}{\partial
t}\right)d\tau\right)dx\,dt\,dy\,dz = 0
$$
while, using the weak-strong convergence theorem in $L^2(\Omega_T)$, we realise that the first term converges to
\begin{equation*}
\left(\int_{Z^*}\rho(z)dz\right) \left(\int_{\Omega_T}
\frac{\partial^2u_0}{\partial t^2}\psi_0(x,t)dx\,dt\right).
\end{equation*}
As for the last term, Lemma~\ref{key4} applies and yields the following limit
\begin{equation*}
\int_{\Omega_T}\left( \int_0^1\!\!\!\int_{Y} \left[
u_2 , \rho\chi_{Z^*}\frac{\partial^2\psi_2}{\partial\tau^2}
\right]dyd\tau \right)dx\,dt.
\end{equation*}
Hence, as $\varepsilon\to 0$, $L_1$ converges to
\begin{equation}\label{eq4.222}
\left(\int_{Z^*}\rho(z)dz\right) \left(\int_{\Omega_T}
\frac{\partial^2u_0}{\partial t^2}\psi_0(x,t)dx\,dt\right) + \int_{\Omega_T}\left( \iint_{Y\times \mathcal{T}} \left[
u_2\rho\chi_{Z^*} , \frac{\partial^2\psi_2}{\partial\tau^2}
\right]d\tau dy\right)dx\,dt.
\end{equation}

Considering now the second term in the left hand side of (\ref{vareq}) which we denote by $L_2$, we have
\begin{align}
L_2 & =\int_{\Omega_T}\beta^\varepsilon \chi_{Z^*}^\varepsilon\frac{\partial (P_\varepsilon
 u_\varepsilon)}{\partial t}\psi_{0}(x,t)dxdt + \varepsilon \int_{\Omega_T}\beta^\varepsilon\chi_{Z^*}^\varepsilon\frac{\partial (P_\varepsilon
 u_\varepsilon)}{\partial t}\psi_{1}\left(x,t,\frac{x}{\varepsilon} \right)dxdt \nonumber \\
  &\qquad +\varepsilon^2\int_{\Omega_T}\beta^\varepsilon\chi_{Z^*}^\varepsilon
\frac{\partial (P_\varepsilon
 u_\varepsilon)}{\partial t}\psi_{2}\left(x,t,\frac{x}{\varepsilon},\frac{x}{\varepsilon^2},\frac{t}{\varepsilon^2}\right)dxdt.\label{eq4.3}
\end{align}
The second and third terms in (\ref{eq4.3}) tend to zero as  $\varepsilon \to 0$. As the first term is concerned, (\ref{eqwms555}) and item (iii) of Remark~\ref{r2} help to get  the following limit
\begin{eqnarray*}
&&\iiiint_{\Omega_T\times Y\times Z\times \mathcal{T}}\beta(y,\tau)\chi_{Z^*}(z)\frac{\partial
u_0}{\partial t}\psi_0dxdtdydzd\tau\\&&\qquad\qquad
=|Z^*|\left(\int_0^1\!\!\!\int_ {Y}\beta(y,\tau)dy\,d\tau\right)\left(
\int_{\Omega_T}\frac{\partial u_0}{\partial
t}\psi_0dxdt\right).
\end{eqnarray*}
Finally, it is classical that the third term in the left hand side of (\ref{vareq}) converges to
\begin{equation}\label{eq4.4}
\iiiint_{\Omega_T\times Y\times Z \times
\mathcal{T}}A(y,z)\left(\nabla_x u_0 + \nabla_y
u_1 + \nabla_z u_2\right) \cdot\chi_{Z^*}(z)\left(\nabla_x \psi_0
  + \nabla_y \psi_1 + \nabla_z
\psi_2\right)\,dxdtdydzd\tau,
\end{equation}
while its right hand side tends to
\begin{equation}\label{eq4.5}
|Z^*| \int_{\Omega_T}f(x,t)\psi_0(x,t)dx\,dt.
\end{equation}
To formulate our global limit problem, we need to prove that $u_2(x,t)\in V$ for almost all $(x,t)\in\Omega_T$, such that we can rewrite the duality bracket in (\ref{eq4.222}) using the formula in Proposition~\ref{prop10}, viz.,
$$
\int_0^1\!\!\!\int_{Y}\left[ u_2 ,\rho\chi_{Z^*}
\frac{\partial^2\psi_2}{\partial\tau^2} \right]dy\,d\tau =
\int_0^1\!\!\!\int_{Y}\left[\rho
\chi_{Z^*}\frac{\partial^2 u_2}{\partial\tau^2}, \psi_2
\right]dy\,d\tau\quad  \text{a.e. \ in }\Omega_T.
$$
\begin{proposition}
The function $u_2 \in L^2(\Omega_T; L_{\text per}^{2}(Y\times \mathcal{T};H_{\#\rho}^{1}(Z^*)))$ defined by Theorem~\ref{theo6} satisfies 
$$
u_2(x,t)\in V \text{ for almost all } (x,t)\in\Omega_T.
$$
\end{proposition}
\begin{proof}
After the passage to the limit in (\ref{vareq}), we obtain
\begin{align*}
 &\left(\int_{Z^*}\rho(z)dz\right)
\int_{\Omega_T}\frac{\partial^2 u_0}{\partial t^2}\psi_0\, dxdt
 + \int_{\Omega_T}\left(\int_0^1\!\!\!\int_{Y}
 \left[u_2, \rho\chi_{Z^*}\frac{\partial^2
\psi_2}{\partial \tau^2}\right]dyd\tau\right)dxdt \\
& + |Z^*|\left( \int_0^1\!\!\!\int_{Y}\beta(y,\tau)dyd\tau \right)
\int_{\Omega_T}\frac{\partial u_0}{\partial t}\psi_0\,dx\,dt\\
& + \iiiint_{\Omega_T\times Y\times
Z^*\times\mathcal{T}}A(y,z)(\nabla_x u_0 + \nabla_y u_1 +
\nabla_{z} u_2)\cdot\left(\nabla_x \psi_0  + \nabla_y \psi_1 +
\nabla_{z}
\psi_2\right)dx\,dt\,dy\,dz\,d\tau \\
&= |Z^*| \int_{\Omega_T}f(x,t)\psi_0(x,t)dx\,dt,
\end{align*}
for all $(\psi_0,\psi_1,\psi_2)\in \mathcal{D}(\Omega_T)\times \left(\mathcal{D}(\Omega_T)\otimes \mathcal{C}^\infty_{\#}(Y)\right)\times \left(\mathcal{D}(\Omega_T)\otimes\mathcal{C}^\infty_{\text per}(Y)\otimes\mathcal{C}^\infty_{\text per}(\mathcal{T})\otimes \mathcal{C}^\infty_{\#\rho}(Z^*)\right)$. Taking in this equation $\psi_0 = \psi_1= 0$ and
$\psi_2 = \varphi\otimes\phi$ where $\varphi \in
\mathcal{D}(\Omega_T)$ and $\phi \in \mathcal{C}^\infty_{\text per}(Y\times
\mathcal{T})\otimes\mathcal{C}^\infty_{\#\rho}(Z^*)$, and using the arbitrariness of $\varphi$, we obtain almost everywhere in $\Omega_T$,
$$
\int_0^1\!\!\! \int_{Y\times \mathcal{T}}\left[
u_2 ,\rho \chi_{Z^*} \frac{\partial^2\phi}{\partial\tau^2}
\right]dy\,d\tau  = - \iiint_{ Y\times Z^* \times
\mathcal{T}}A(y,z)\left(\nabla_x u_0 + \nabla_{y} u_1 +
\nabla_z u_2\right)\cdot\left( \nabla_{z}\phi\right)dy\,dz\,d\tau.
$$
Let $(x,t)\in\Omega_T$ and consider the linear functional
\begin{align*}
\phi &\mapsto -\iiint_{Y\times Z^* \times \mathcal{T
}}A(y,z)(\nabla _{x}u_{0}+\nabla _{y}u_1+\nabla
_{z}u_{2})(\nabla _{z}\phi)\,dy\,dz\,d\tau.
\end{align*}
It is easy to see that the boundedness of the matrix $A$ implies
that the above linear functional is continuous on $\mathcal{C}^\infty_{\text per}(Y\times
\mathcal{T})\otimes\mathcal{C}^\infty_{\#\rho}(Z^*)$ for
the $L^2_{\text per}(Y\times \mathcal{T}, H^1_{\#\rho}(Z^*))$-norm. The Proposition~\ref{prop8} applies and gives \ $\rho
\chi_{Z^*}\frac{\partial^2u_2}{\partial\tau^2} \in L^2_{\text
per}(Y\times \mathcal{T}, (H^1_{\#\rho}(Z^*))') $  almost everywhere in $\Omega_T$. This completes the proof.
\end{proof}
The passage to the limit in (\ref{vareq}) as $E'\ni\varepsilon\to 0$ has proved the following result.
\begin{proposition}The triplet $(u_0,u_1,u_2)$ defined by Theorem~\ref{theo6} is a solution to the following variational problem:
\begin{equation}\left\{
\begin{aligned} \label{eqghp}
&(u_0, u_1, u_2)\in L^2(0,T;H^1_0(\Omega))\times
L^2(\Omega_T; H^1_{\#}(Y))\times
L^2(\Omega_T;L^2(Y\times \mathcal{T};H^1_{\#\rho}(Z^*))) \\
&\left(\int_{Z^*}\rho(z)dz\right)
\int_{\Omega_T}\frac{\partial^2 u_0}{\partial t^2}\psi_0\,dxdt
 + \int_{\Omega_T}\left(\iint_{Y\times\mathcal{T}}
 \left[\rho\chi_{Z^*}\frac{\partial^2
u_2}{\partial \tau^2},\psi_2\right]dyd\tau\right)\,dxdt \\
& + |Z^*|\left( \iint_{Y\times\mathcal{T}}\beta(y,\tau)dyd\tau \right)
\int_{\Omega_T}\frac{\partial u_0}{\partial t}\psi_0\,dxdt\\
&  +\iiiint_{\Omega_T\times Y\times
Z^*\times\mathcal{T}}A(y,z)(\nabla_x u_0 + \nabla_y u_1 +
\nabla_{z} u_2)\cdot\left(\nabla_x \psi_0  + \nabla_y \psi_1 +
\nabla_{z}
\psi_2\right)\,dxdtdydzd\tau \\
&= |Z^*| \int_{\Omega_T}f(x,t)\psi_0(x,t)dxdt\\
& \text{for all } (\psi_0,\psi_1,\psi_2)\in \mathcal{D}(\Omega_T)\times \left(\mathcal{D}(\Omega_T)\otimes \mathcal{C}^\infty_{\#}(Y)\right)\times \left(\mathcal{D}(\Omega_T)\otimes\mathcal{C}^\infty_{\text per}(Y)\otimes\mathcal{C}^\infty_{\text per}(\mathcal{T})\otimes \mathcal{C}^\infty_{\#\rho}(Z^*)\right).
\end{aligned}
\right.
\end{equation}
\end{proposition}
The variational problem (\ref{eqghp}), sometimes called global limit problem, is 
equivalent to the following system of three problems (\ref{eqghp3})-(\ref{eqghp2})-(\ref{eqghp1}):
\begin{equation}  \label{eqghp3}
\left\{\begin{aligned}
&\int_{\Omega_T}\left(\int_0^1\!\!\!\int_{Y}\left[\rho\chi_{Z^*}\frac{\partial^2
u_2}{\partial \tau^2},\psi_2\right]dy\,d\tau\right)\,dx\,dt\\
&+\iiiint_{\Omega_T\times Y\times
Z^*\times\mathcal{T}}A(y,z)(\nabla_x u_0 + \nabla_y u_1 +
\nabla_{z} u_2)\cdot\left(\nabla_{z}
\psi_2\right)\,dx\,dt\,dy\,dz\,d\tau =0\\
& \quad \text{ for all }\psi_2\in
\mathcal{D}(\Omega_T)\otimes\mathcal{C}^\infty_{\text per}(Y)\otimes \mathcal{C}^\infty_{\text per}(\mathcal{T})\otimes \mathcal{C}^\infty_{\#\rho}(Z^*);
\end{aligned}\right.
\end{equation}

\begin{equation}  \label{eqghp2}
\left\{\begin{aligned}
&\iiiint_{\Omega_T\times Y\times
Z^*\times\mathcal{T}}A(y,z)(\nabla_x u_0 + \nabla_y u_1 +
\nabla_{z} u_2)\cdot\left(\nabla_{z}
\psi_1\right)\,dx\,dt\,dy\,dz\,d\tau
 =0\\
 & \quad \text{for all }\psi_1\in
\mathcal{D}(\Omega_T)
\otimes\mathcal{C}^\infty_{\#}(Y);
\end{aligned}\right.
\end{equation}
and
\begin{equation}  \label{eqghp1}
\left\{\begin{aligned}
&\left(\int_{Z^*}\rho(z)dz\right)
\int_{\Omega_T}\frac{\partial^2 u_0}{\partial t^2}\psi_0\,dxdt +
|Z^*|\left( \int_0^1\!\!\!\int_{Y}\beta(y,\tau)dyd\tau \right)
\int_{\Omega_T}\frac{\partial u_0}{\partial t}\psi_0\,dxdt \\ &+
\iiiint_{\Omega_T\times Y\times
Z^*\times\mathcal{T}}A(y,z)(\nabla_x u_0 + \nabla_y u_1 +
\nabla_{z} u_2)\cdot\left(\nabla_x
\psi_0\right)\,dx\,dt\,dy\,dz\,d\tau \\
 &= |Z^*|
\int_{\Omega_T}f(x,t)\psi_0(x,t)dx\,dt \ \ \  \ \ \text{for all }
\psi_0\in  \mathcal{D}(\Omega_T).
\end{aligned}\right.
\end{equation}

We are now in a position to derive equations describing the
microscopic, the mesoscopic and the macroscopic behaviours of the phenomenon modelled by (\ref{eq1}). We start at the microscopic scale.
\subsection{The microscopic problem}
Taking in (\ref{eqghp3}),  $\psi_2= \varphi\otimes\theta\otimes\phi$, with
$\varphi \in \mathcal{D}(\Omega_T)$, $\theta \in \mathcal{C}^\infty_{\text per}(Y)$  and $\phi\in  \mathcal{C}^\infty_{\#\rho}(Z^*)\otimes \mathcal{C}^\infty_{\text per}(\mathcal{T}) $ we
obtain by the arbitrariness of $\varphi$ and $\theta$ and for almost all $(x,t,y)\in \Omega_T\times Y$,
$$
\int_0^1\left[\rho\chi_{Z^*}\frac{\partial^2
u_2}{\partial \tau^2},\phi\right] d\tau = -\int_0^1\!\!\!\int_{Z^*}A(y,z)(\nabla_x u_0 + \nabla_y u_1 +
\nabla_{z} u_2)\cdot\left(\nabla_{z}
\phi\right)dzd\tau.
$$
Therefore, for almost all $(x,t,y)\in \Omega_T\times Y$, the function $u_2(x,t,y)$ solves the following variational problem
\begin{equation}\label{eqva1}
\left\{\begin{aligned}
& u_2(x,t,y) \in L^2_{\text per}(\mathcal{T};H^1_{\#\rho}(Z^*)), \\
&\int_0^1\left[\rho\chi_{Z^*}\frac{\partial^2
u_2}{\partial \tau^2},v\right]d\tau +  \int_0^1\!\!\!\int_{Z^*}A(y,z)(\nabla_{z}u_2)\cdot\left(\nabla_{z}
v\right)dzd\tau \\&\qquad\qquad = -\int_0^1\!\!\!\int_{Z^*}A(y,z)(\nabla_x u_0 + \nabla_y u_1)\cdot\left(\nabla_{z}
v\right)dzd\tau\\
& \text{for all }\  v \in   L^2_{\text per}(\mathcal{T};H^1_{\#\rho}(Z^*)).
\end{aligned}\right.
\end{equation}
Moreover, the variational problem (\ref{eqva1}) admits a solution uniquely defined on $Z^*\times\mathcal{T}$ since, if $v_1, v_2\in  L^2_{\text per}(\mathcal{T};H^1_{\#\rho}(Z^*))$ are two solutions,  then $v= v_1-v_2\in  L^2_{\text per}(\mathcal{T};H^1_{\#\rho}(Z^*))$ will be solution of the linear homogeneous equation
$$
\rho(z)\frac{\partial^2
v}{\partial \tau^2}+\text{div}_{z}\left(A(y,z)(\nabla_{z}v)\right) = 0 \qquad \text{in }\ \
Z^*\times\mathcal{T}
$$
with zero Cauchy data, we therefore deduce  that $v=0$  in $Z^*\times\mathcal{T}$ and \textit{uniqueness} follows. 

As customary, let $y\in Y$ be fixed and let  $ \chi_i(y)$ $(1\leq i\leq N)$ be the \textit{unique} solution to the microscopic problem
\begin{equation}\label{eq4.13}
\left\{\begin{aligned}
&\chi_i(y)\in L^2_{\text per}(\mathcal{T};H^1_{\#\rho}(Z^*)),\\
& \int_0^1\left[\rho\chi_{Z^*}\frac{\partial^2
\chi_i}{\partial \tau^2},\phi\right]d\tau +\int_0^1\!\!\!  \int_{
Z^*}A(y,z)(\nabla_{z}\chi_i)\cdot\left(\nabla_{z}
\phi\right)dzd\tau\\
 & = -\int_0^1\!\!\!  \int_{
Z^*}\sum_{1=1}^N a_{ik}(y,z)\frac{\partial \phi}{\partial z_k}dzd\tau,\\
&\text{ for all }\phi \in L^2_{\text per}(\mathcal{T};H^1_{\#\rho}(Z^*))\qquad (i\in\{1,\cdots, N\}).
\end{aligned}
\right.
\end{equation}
Multiplying the $\text{i}^{th}$ equation of (\ref{eq4.13}) by $\frac{\partial u_0}{\partial x_i}+\frac{\partial u_1}{\partial y_i}$ and then summing the resulting equations over $i=1,\cdots, N$, it appears that the function $(x,t,y,z,\tau)\mapsto \chi(y, z,\tau)\cdot \left( \nabla_x u_0 (x,t) + \nabla_{y}u_1 (x,t,y)\right) $   is a solution to (\ref{eqva1}). Hence, setting $\chi=(\chi_i)_{1\leq i\leq N}$ it holds almost everywhere in $\Omega_T\times Y\times Z^*\times\mathcal{T}$ that
\begin{equation}\label{eq4.14}
 u_2(x,t,y,z,\tau) = \chi(y, z,\tau)\cdot \left( \nabla_x u_0 (x,t) + \nabla_{y}u_1 (x,t,y) \right).
 \end{equation}
On putting
\begin{equation*}
(\nabla_{z}\chi)_{ij}=\frac{\partial \chi_i}{\partial z_j}\qquad
(1\leq i,j \leq N),
\end{equation*}
we can deduce from (\ref{eq4.14}) that
\begin{equation}  \label{eqmi}
\nabla_{z} u_2 = \nabla_{z}\chi\cdot(\nabla_x u_0 + \nabla_{y} u_1)\qquad {\text a.e.\ in\ } \Omega_T\times Y\times Z^*\times\mathcal{T}.
\end{equation}

\subsection{The mesoscopic problem} Taking $\psi_1 = \varphi\otimes \phi\otimes\theta$ with $\varphi\in \mathcal{D}(\Omega_T)$, $\theta \in \mathcal{C}^\infty_{\text per}(\mathcal{T})$ and   $\phi \in \mathcal{C}^\infty_{\#}(Y)$ in (\ref{eqghp2}) and using (\ref{eqmi}) and the arbitrariness of $\varphi$ and $\theta$, we realise that for almost every $(x,t)\in \Omega_T$, the function $u_1(x,t)$ is the unique solution to the following variational problem (where $I$ denotes the $N\times N$ identity matrix)
\begin{equation}\label{eq4.16}
\left\{\begin{aligned}
&u_1(x,t)\in H^1_{\#}(Y),\\
& \int_{ Y}\!\!\left(\int_0^1\!\!\!  \int_{
Z^*}\!\!A(I+ \nabla_{z}\chi)dzd\tau\right)  \nabla_y u_1 \cdot \nabla_{y}
\phi\, dy
= - \int_{ Y}\!\!\left(\int_0^1\!\!\!  \int_{
Z^*}\!\!A(I+ \nabla_{z}\chi)dzd\tau\right)  \nabla_x u_0\cdot \nabla_{y}
\phi\, dy\\ &\text{  for all }\ \phi \in H^1_{\#}(Y).
\end{aligned}\right.
\end{equation}
For the sake of shortness, we put
$$
\tilde{A}(y)= \int_0^1\!\!\!  \int_{
Z^*}A(y,z)(I+ \nabla_{z}\chi)dzd\tau\qquad (y\in Y),
$$
and recall that there exists a unique  $\theta = (\theta_i)_{1\leq i\leq N}\in \left(H^1_{\#}(Y)\right)^N$ solution to the mesoscopic problem
\begin{equation}  \label{eqme}
\left\{\begin{aligned}
&\theta_i\in H^1_{\#}(Y),\\&\int_{Y}\tilde{A}  \nabla_y\theta_i\cdot \nabla_{y}
v\,dy
 =- \int_{Y} \sum_k \tilde{a}_{ik}\frac{\partial v}{\partial y_k}  dy\\
& \text{for all }\  v \in H^1_{\#}(Y)\qquad \qquad\qquad\qquad \qquad\qquad(i=1,\cdots,N).
\end{aligned}\right.
\end{equation}
It is easy to check that the function $(x,t,y)\mapsto\theta(y)\cdot\nabla_x u_0(x,t)$ is also a solution to (\ref{eq4.16}) so that by the uniqueness of the solution to (\ref{eq4.16}) we have
$$
u_1(x,t,y) = \theta(y)\cdot\nabla_x u_0(x,t)\qquad \text{for a.e. }(x,t,y)\in\Omega_T\times Y.
$$
On setting
$$    \left( \nabla_y\theta \right)_{ij} =\frac{\partial \theta_i}{\partial y_j} \qquad (1\leq i,j\leq N),
$$
it follows that
\begin{equation}\label{eq4.18}
\nabla_y u_1(x,t,y)= \nabla_y\theta(y)\cdot\nabla_x u_0(x,t)\quad  \text{a.e. in  } \Omega_T\times Y.
\end{equation}

\subsection{The macroscopic problem}As far as (\ref{eqghp1}) is concerned, we make use of (\ref{eqmi}) and (\ref{eq4.18}) to write
 $$
 \nabla_x u_0 + \nabla_zu_1 +
\nabla_{z} u_2 = \left(I+ \nabla_{z}\chi \right)\left(I + \nabla_y\theta \right)\nabla_x  u_0\ \ \ \text{a.e. in }\ \Omega_T\times Y\times Z^*\times\mathcal{T}
 $$
where the functions $ \chi$ and  $\theta $ are the solutions to the problems (\ref{eq4.13}) and (\ref{eqme}), respectively.  We have
\begin{equation}
\begin{aligned}
&\iiiint_{\Omega_T\times Y\times
Z^*\times\mathcal{T}}A(y,z)(\nabla_x u_0 + \nabla_y u_1 +
\nabla_{z} u_2)\cdot\left(\nabla_x
\psi_0\right)\,dx\,dt\,dy\,dz\,d\tau \\
&= \iint_{\Omega_T\times Y }\left(\int_0^1\!\!\!  \int_{
Z^*}\!A(y,z)\left(I+ \nabla_{z}\chi \right)dz d\tau \right)\left(I + \nabla_y\theta \right)\nabla_xu_0\cdot\nabla_x\psi_0\, dxdtdy\\
&=\iint_{\Omega_T\times Y}\tilde{A}(y)(I + \nabla_y\theta)\nabla_x u_0\cdot\nabla_x\psi_0\, dxdtdy\\
&=\int_{\Omega_T}\left(\int_{Y}\tilde{A}(y)(I + \nabla_y\theta)dy\right)\nabla_xu_0\cdot\nabla_x\psi_0\, dxdt\\
&=\int_{\Omega_T}\hat{A}\nabla_x u_0\cdot\nabla_x\psi_0\, dxdt,
\end{aligned}
\end{equation}
where $ \hat{A}= \int_{Y}\tilde{A}(I + \nabla_y\theta)dy$. With this notation, the variational problem (\ref{eqghp1}) implies
\begin{equation}  \label{eqghp12}
\left\{\begin{aligned}
&\left(\int_{Z^*}\rho(z)dz\right)
\int_{\Omega_T}\frac{\partial^2 u_0}{\partial t^2}\psi_0\,dxdt +
|Z^*|\left( \int_0^1\!\!\!  \int_{
Y}\beta(y,\tau)dyd\tau \right)
\int_{\Omega_T}\frac{\partial u_0}{\partial t}\phi\,dxdt \\ &-
\int_{\Omega_T}\operatorname{div}\left(\hat{A}\nabla_x u_0 \right)\psi_0 dxdt \\
 &= |Z^*|
\int_{\Omega_T}f(x,t)\psi_0(x,t)dxdt \ \ \  \ \ \text{for all }
\psi_0\in  \mathcal{D}(\Omega_T),
\end{aligned}\right.
\end{equation}
which is nothing but the weak formulation of
\begin{equation}  \label{eqghp13}
\mathcal{M}_{Z^*}(\rho)
\frac{\partial^2 u_0}{\partial t^2} +
\mathcal{M}_{Y\times \mathcal{T}}(\beta)
\frac{\partial u_0}{\partial t}
-\frac{1}{|Z^*|}\operatorname{div}\left(\hat{A}\nabla u_0 \right) =f(x,t)\quad\text{in }\ \Omega_T.
\end{equation}
We are almost done with the proof of the following theorem which is the main result of the paper.

\begin{theorem}
Assume that the hypotheses \textbf{A1-A3} hold and let $u_\varepsilon$ $(\varepsilon>0)$ be the unique solution to (\ref{eq1}). Let $u_0$ be the function  defined by Theorem~\ref{theo6}   and solution to the variational problem (\ref{eqghp1}).  Then as $0< \varepsilon\to 0$ we have
\begin{equation}\label{eq4.22}
u_\varepsilon \to u_0 \qquad \text{ in }\quad L^2(\Omega_T),
\end{equation}

where $u_0\in L^2(0,T; H^1_0(\Omega))$ with $\frac{\partial u_0}{\partial t}\in L^2(0,T:L^2(\Omega))$, is the unique solution to\begin{equation}\label{eqfinal}
\left\{
\begin{aligned}
\mathcal{M}_{Z^*}(\rho)
\frac{\partial^2 u_0}{\partial t^2} +
\mathcal{M}_{Y\times\mathcal{T}}(\beta)
\frac{\partial u_0}{\partial t}
-\frac{1}{|Z^*|}\operatorname{div}\left(\hat{A}\nabla_xu_0 \right)
& = f(x,t) \quad\text{ in }\  \Omega\times (0,T),\\
u_{0} &=  \ 0   \quad\text{on }\  \partial\Omega\times (0,T),\\
 u_{0}(x,0) & = \ u^{0}(x) \quad\text{in }\ \Omega,\\
 \mathcal{M}_{Z^*}(\rho)\frac{\partial u_0}{\partial t}(x,0)& =\mathcal{M}_{Z^*}(\sqrt{\rho}) v^0(x) \quad\text{in }\ \Omega.
 \end{aligned}
\right.
\end{equation}
\end{theorem}
\begin{proof}
The arbitrariness of the fundamental sequence $E$ in the limit passage in this section and the uniqueness of the solution to (\ref{eqfinal}) prove that we have (\ref{eq4.22}) for the whole  generalised sequence $\varepsilon>0$. Hence, it remains to justify the initial conditions appearing in the macroscopic problem (\ref{eqfinal}).
We start by justifying that $u_0(x,0)= u^0(x)$ for almost every $x\in\Omega$. This is obvious since $u_0, P_\varepsilon u_\varepsilon\in C([0,T]; L^2(\Omega))$\ ($\varepsilon>0 $) and $P_\varepsilon u_\varepsilon \to u_0 $ strongly in $ L^2(\Omega_T)$ with $P_\varepsilon u_\varepsilon(x,0)= u^0(x)$.

Next, we justify the initial condition satisfied by  $\frac{\partial u_0}{\partial t}(x,0)$. 
We consider a function $\phi \otimes\varphi$ where $\phi \in \mathcal{D}(\Omega)$ and
$\varphi \in \mathcal{D}([0,T])$ with $\varphi(T)= 0$ and $\varphi(0)=1$. 
After  multiplying  the main equation in (\ref{eq1}) by  $\phi(x) \varphi(t)$, we integrate over $\Omega^\varepsilon_T$ and perform an integration by parts with respect to the variable $t$ in the integral containing the term $\frac{\partial^2 u_\varepsilon}{\partial t^2}\varphi$ (with the initial condition $ \rho(\frac{x}{\varepsilon^2})\frac{\partial u_\varepsilon}{\partial t}(x,0)= \rho^{\frac{1}{2}}(\frac{x}{\varepsilon^2})v^0(x)$ in mind), we obtain:
\begin{equation*}
\begin{aligned}
&-\int_{\Omega^\varepsilon} \rho^{\frac{1}{2}}\left(\frac{x}{\varepsilon^2}\right)v^0(x)\phi(x)dx
-\int_{\Omega^\varepsilon_T} \rho\left(\frac{x}{\varepsilon^2}\right)\phi(x)\frac{\partial u_\varepsilon}{\partial t}\varphi'(t)dxdt\\ & \qquad
+\int_{\Omega^\varepsilon_T}\beta\left( \frac{x}{\varepsilon}, \frac{t}{\varepsilon^2}\right)\frac{\partial u_\varepsilon}{\partial t}\phi(x)\varphi(t)dxdt 
 + \int_{\Omega^\varepsilon_T}A\left(\frac{x}{\varepsilon
}, \frac{x}{\varepsilon ^{2}}\right)\varphi(t) \nabla u_{\varepsilon }.\nabla\phi(x) dxdt\\& \qquad\qquad = \int_{\Omega^\varepsilon_T}f(x,t)\phi(x)\varphi(t)dxdt. 
\end{aligned}
\end{equation*}
When $0<\varepsilon \to 0$, using the same arguments  as in the derivation of the global limit problem , we obtain 
\begin{equation}\label{eq4.25}
\begin{aligned}
&-\mathcal{M}_{Z^*}(\sqrt{\rho})\int_{\Omega}v^0(x)\phi(x)dx
-\mathcal{M}_{Z ^*}(\rho)\int_{\Omega_T}\frac{\partial u_0}{\partial t}\varphi'(t)\phi(x)dxdt \\ &\qquad 
+\mathcal{M}_{Y\times\mathcal{T}}(\beta)\int_{\Omega_T}\frac{\partial u_0}{\partial t}\varphi(t)\phi(x)dxdt- \frac{1}{|Z^*|}
\int_{\Omega_T}\operatorname{div}\left(\hat{A}\nabla u_0 \right)\phi(x)\varphi(t)dxdt\\ &\qquad\qquad =\int_{\Omega_T}f(x,t)\phi(x)\varphi(t)dxdt.
\end{aligned}
\end{equation}
Keeping (\ref{eqfinal}) in mind, an integration by parts with respect to the variable $t$ in the second term of (\ref{eq4.25}) yields
\begin{equation}
\int_{\Omega}\left(\mathcal{M}_{Z^*}(\rho)\frac{\partial u_0}{\partial t}(x,0)-\mathcal{M}_{Z^*}(\sqrt{\rho})v^0(x)\right)\phi(x)dx=0,
\end{equation}
which implies 
$$
\mathcal{M}_{Z^*}(\rho)\frac{\partial u_0}{\partial t}(x,0) =\mathcal{M}_{Z^*}(\sqrt{\rho}) v^0(x) \quad\text{in }\ \Omega.
$$
The proof is completed. 
\end{proof}

\end{document}